\documentclass{article}%
\usepackage{amsmath}
\usepackage{graphicx}
\usepackage{amsfonts}
\usepackage{amssymb}%
\providecommand{\U}[1]{\protect\rule{.1in}{.1in}}
\newtheorem{theorem}{Theorem}[section]

\newtheorem{corollary}{Corollary}[section]

\newtheorem{lemma}{Lemma}[section]
\newtheorem{proposition}{Proposition}[section]
\newtheorem{remark}{Remark}[section]

\newenvironment{proof}[1][Proof]{\textbf{#1.} }{\ \rule{1em}{1em}}
\begin{document}

\title{Samll BGK waves and nonlinear Landau damping (higher dimensions)}
\author{Zhiwu Lin and Chongchun Zeng\\School of Mathematics\\Georgia Institute of Technology\\Atlanta, GA 30332, USA}
\date{}
\maketitle

\begin{abstract}
Consider Vlasov-Poisson system with a fixed ion background and periodic
condition on the space variables, in any dimension $d\geq2$. First, we show
that for general homogeneous equilibrium and any periodic $x-$box, within any
small neighborhood in the Sobolev space $W_{x,v}^{s,p}\ \left(  p>1,s<1+\frac
{1}{p}\right)  \ $of the steady distribution function, there exist nontrivial
travelling wave solutions (BGK waves) with arbitrary traveling speed. This
implies that nonlinear Landau damping is not true in $W^{s,p}\left(
s<1+\frac{1}{p}\right)  \ $space for any homogeneous equilibria and in any
period box. The BGK waves constructed are one dimensional, that is, depending
only on one space variable. Higher dimensional BGK waves are shown to not
exist. Second, for homogeneous equilibria satisfying Penrose's linear
stability condition, we prove that there exist no nontrivial invariant
structures in the $\left(  1+\left\vert v\right\vert ^{2}\right)  ^{b}%
$-weighted $H_{x,v}^{s}$ $\left(  b>\frac{d-1}{4},\ s>\frac{3}{2}\right)  $
neighborhood. Since arbitrarilly small BGK waves can also be constructed near
any homogeneous equilibria in such weighted $H_{x,v}^{s}$ $\left(  s<\frac
{3}{2}\right)  $ norm, this shows that $s=\frac{3}{2}$ is the critical
regularity for the existence of nontrivial invariant structures near stable
homogeneous equilibria. These generalize our previous results in the one
dimensional case.

\end{abstract}

\section{Introduction}

Consider a collisionless electron plasma with a fixed homogeneous neutralizing
ion background. The Vlasov-Poisson system in $d$ dimension is
\begin{subequations}
\label{vpe}
\begin{equation}
\partial_{t}f+v\cdot\nabla_{x}f-\vec{E}\cdot\nabla_{v}f=0, \label{vlasov}%
\end{equation}
\begin{equation}
\,\,\,\,E=-\nabla_{x}\phi,\ \ -\Delta\phi=-\int_{\mathbf{R}^{d}%
}f\ dv+1,\label{poisson}%
\end{equation}
\end{subequations}
where $f\left(  t,x,v\right)  \geq0$ is the distribution function, $E\left(
x,t\right)  $ is the electrical field and $\phi\left(  x,t\right)  $ is the
electrical potential. We consider the Vlasov-Poisson system in a $x-$periodic
box, with periods $T_{i}$ in $x_{i}$. In 1946, Landau \cite{landau}, looking
for analytical solutions of the linearized Vlasov-Poisson system around
Maxwellian $\left(  e^{-\frac{1}{2}v^{2}},0\right)  $, pointed out that the
electric field is subject to time decay even in the absence of collisions. The
effect of this Landau damping, as it is subsequently called, plays a
fundamental role in the study of plasma physics. However, Landau's treatment
is in the linear regime; that is, only for infinitesimally small initial
perturbations. Recently, nonlinear Landau damping was shown (\cite{villani09})
for analytical perturbations of stable equilibria with linear exponential
decay. For general perturbations in Sobolev spaces, the proof of nonlinear
damping remains open. We refer to \cite{lin-zeng-bgk} \cite{villani09} for
more discussions and references on this topic. In \cite{lin-zeng-bgk}, the
following results were obtained for $1$D Vlasov-Poisson system: First, we show
that for general homogeneous equilibria, within any small neighborhood in the
Sobolev space $W^{s,p}\ \left(  p>1,s<1+\frac{1}{p}\right)  \ $of the steady
distribution function, there exist nontrivial travelling wave solutions (BGK
waves) with arbitrary minimal period and traveling speed. This implies that
nonlinear Landau damping is not true in $W^{s,p}\left(  s<1+\frac{1}%
{p}\right)  \ $space for any homogeneous equilibria and any spatial period.
Second, it is shown that for homogeneous equilibria satisfying Penrose's
linear stability condition, there exist no nontrivial travelling BGK waves and
unstable homogeneous states in some $W^{s,p}$ $\left(  p>1,s>1+\frac{1}%
{p}\right)  \ $neighborhood. Furthermore, we prove that there exist no
nontrivial invariant structures in the $H^{s}$ $\left(  s>\frac{3}{2}\right)
$ neighborhood$~$of stable homogeneous states. In particular, these results
suggest the contrasting long time dynamics in the $H^{s}$ $\left(  s>\frac
{3}{2}\right)  $ and $H^{s}$ $\left(  s<\frac{3}{2}\right)  \ $neighborhoods
of a stable homogeneous state.

In this paper, we generalize above results to higher dimensions ($d=2,3$).
Denote the fractional order Sobolev spaces$\ $by $W^{s,p}\left(
\mathbf{R}^{d}\right)  \ $or $W_{x_{1},v}^{s,p}\left(  \left(  0,T_{1}\right)
\times\mathbf{R}^{d}\right)  $ with $p>1,s\geq0$. These spaces are the complex
interpolation of of $L^{p}$ space and Sobolev spaces $W^{m,p}$ $\left(
m\ \text{positive integer}\right)  $. Our first result is to construct ($1$D)
BGK waves in $W_{x_{1},v}^{s,p}\ \left(  s<1+\frac{1}{p}\right)  \ $spaces.

\begin{theorem}
\label{thm-existence}Assume the homogeneous distribution function
\[
f_{0}\left(  v\right)  \in W^{s,p}\left(  \mathbf{R}^{d}\right)
\ \ \ \ \left(  d\geq2,\ p>1,s\in\lbrack0,1+\frac{1}{p})\right)
\]
satisfies%
\[
f_{0}\left(  v\right)  \geq0,\ \int f_{0}\left(  v\right)  dv=1,\ \int
v^{2}f_{0}\left(  v\right)  dv<+\infty.
\]
Fix $T_{1}>0$ and $c\in\mathbf{R}$. Then for any $\varepsilon>0$, there exist
travelling BGK wave solutions of the form $f=f_{\varepsilon}\left(
x_{1}-ct,v\right)  ,$ $\vec{E}=E_{\varepsilon}\left(  x_{1}-ct\right)  \vec
{e}_{1}$ to (\ref{vpe}), such that$\ \left(  f_{\varepsilon}\left(
x_{1},v\right)  ,E_{\varepsilon}\left(  x_{1}\right)  \right)  $ has minimal
period $T_{1}$ in $x_{1}$,$\ f_{\varepsilon}\left(  x_{1},v\right)  \geq0,$
$E_{\varepsilon}\left(  x_{1}\right)  $ is not identically zero, and
\begin{equation}
\ \left\Vert f_{\varepsilon}-f_{0}\right\Vert _{L_{x_{1},v}^{1}}+\ \int
_{0}^{T_{1}}\int_{\mathbf{R}^{d}}\left\vert v\right\vert ^{2}\left\vert
f_{\varepsilon}\left(  x_{1},v\right)  -\ f_{0}\left(  v\right)  \right\vert
dx_{1}dv+\left\Vert f_{\varepsilon}-f_{0}\right\Vert _{W_{x_{1},v}^{s,p}%
}<\varepsilon.\ \label{norm-thm-stability}%
\end{equation}

\end{theorem}

In Proposition \ref{prop-non-bgk}, we show that there exist no $2$D and $3$D
BGK waves. Therefore, the form of $1$D BGK waves in Theorem
\ref{thm-existence} is somehow necessary.

For any $b>\frac{d-1}{4},\ $we denote $H_{v}^{s,b}\left(  \mathbf{R}%
^{d}\right)  $ to be the $\left(  1+\left\vert v\right\vert ^{2}\right)
^{b}\ $weighted $H^{s}$ space, that is,
\begin{equation}
H_{v}^{s,b}\left(  \mathbf{R}^{d}\right)  =\left\{  f\ |\left\Vert \left(
1+\left\vert v\right\vert ^{2}\right)  ^{b}\left(  1-\Delta\right)  ^{\frac
{s}{2}}f\right\Vert _{L^{2}\left(  \mathbf{R}^{d}\right)  }<\infty\right\}  .
\label{definition-H-tilde}%
\end{equation}
and
\[
\left\Vert f\right\Vert _{H_{v}^{s,b}}=\left\Vert \left(  1+\left\vert
v\right\vert ^{2}\right)  ^{b}\left(  1-\Delta\right)  ^{\frac{s}{2}%
}f\right\Vert _{L^{2}\left(  \mathbf{R}^{d}\right)  }.
\]
Let $\mathbf{T}^{d}$ be a periodic box with periods $T_{i}\ $in $x_{i}$
$\left(  i=1,\cdots,d\right)  $, and
\[
\mathbf{Z}^{d}=\left\{  \left(  \frac{2\pi}{T_{1}}j_{1},\cdots,\frac{2\pi
}{T_{d}}j_{d}\right)  \ |\ j_{1},\cdots,j_{d}\text{ are integers}\right\}  .
\]
We define the space $H_{x}^{s_{x}}H_{v}^{s_{v},b}\left(  \mathbf{T}^{d}%
\times\mathbf{R}^{d}\right)  $ by%
\[
h=\sum_{\vec{k}\in\mathbf{Z}^{d}}e^{i\vec{k}\cdot x}h_{\vec{k}}\left(
v\right)  \in H_{x}^{s_{x}}H_{v}^{s_{v},b}%
\]
if $\ $
\[
\text{ }\left\Vert h\right\Vert _{H_{x}^{s_{x}}H_{v}^{s_{v},b}}=\left(
\left\Vert h_{\vec{0}}\right\Vert _{H_{v}^{s_{v},b}}^{2}+\sum_{0\neq\vec{k}%
\in\mathbf{Z}^{d}}\left\vert \vec{k}\right\vert ^{2s_{x}}\left\Vert h_{\vec
{k}}\right\Vert _{H_{v}^{s_{v},b}}^{2}\right)  ^{\frac{1}{2}}<\infty.
\]
The following Theorem excludes any nontrivial invariant structures (steady,
time periodic, quasi-periodic etc) near stable homogeneous equilibria in the
$H_{x}^{s_{x}}H_{v}^{s_{v},b}\ $spaces of high $v-$regularity.

\begin{theorem}
\label{thm-non-existence}Consider the homogeneous profile
\begin{equation}
f_{0}\left(  v\right)  \in H^{s_{0},b}\left(  \mathbf{R}^{d}\right)
\ \ \ \left(  d\geq2,\ s_{0}>\frac{3}{2},\ b>\frac{d-1}{4}\right)  ,
\label{condition f_0}%
\end{equation}
Let $T^{d}$ be a periodic box with periods $T_{i}\ $in $x_{i}$ $\left(
i=1,\cdots,d\right)  .\ $Assume that $f_{0}\left(  v\right)  $ satisfies the
Penrose stability condition (\ref{penrose-condition}) for $\left(
T_{1},\cdots,T_{d}\right)  $. Let $\left(  f\left(  x,v,t\right)  ,\vec
{E}\left(  x,v,t\right)  \right)  $ be a solution of (\ref{vpe}) in $T^{d}.$

For any $\left(  s_{x},s_{v}\right)  $ satisfying
\begin{equation}
s_{x}\geq0,\ s_{x}>\frac{d-3}{2},\ \text{and \ }\frac{3}{2}<s_{v}\leq s_{0},
\label{condition-non-existence}%
\end{equation}
$\ $there exists $\varepsilon_{0}>0$, such that if
\begin{equation}
\left\Vert f\left(  t\right)  -f_{0}\right\Vert _{H_{x}^{s_{x}}H_{v}^{s_{v}%
,b}}<\varepsilon_{0},\ \text{for\ all\ }t\in\mathbf{R},
\label{assumption-thm-invariant-2d}%
\end{equation}
then $\vec{E}\left(  t\right)  \equiv\vec{0}$ for all $t\in\mathbf{R}.$
\end{theorem}

In the above Theorem, the assumption $s_{x}>\frac{d-3}{2}$ is to make
$H_{x}^{s_{x}}$ an algebra which would be needed in the proof of Lemma
\ref{lemma-estimate-integral-small}. The use of weighted Sobolev space
$H_{v}^{s,b}$ in Theorem \ref{thm-non-existence} is rather natural in higher
dimensions. Indeed, even to state the Penrose's stability condition
(\ref{penrose-condition}), we need to assume that the homogeneous equilibrium
$f_{0}\left(  v\right)  \in H^{s_{0},b}\left(  \mathbf{R}^{d}\right)  $ with
$\left(  s_{0},b\right)  \ $satisfying (\ref{condition f_0}). This is because
that linear instability (stability) of homogeneous equilibria of
Vlasov-Poisson is longitudinal along the wave direction of perturbation. The
weighted Sobolev space (\ref{condition f_0}) is needed to ensure that the
projected steady distribution function in any wave direction is in
$H^{s}\left(  \mathbf{R}\right)  $ $\left(  s>\frac{3}{2}\right)  $ which is
necessary to get the $1$D Penrose stability criterion. Moreover, in Theorem
\ref{Thm-existence-weighted} we also construct ($1$D) BGK waves arbitrarily
near any homogeneous equilibrium in $H_{x}^{s_{x}}H_{v}^{s_{v},b}%
(d\geq2,\ b>\frac{d-1}{4})\ $\ for any $s_{x}>0$ and $s_{v}<\frac{3}{2}$.
Combined with Theorem \ref{thm-non-existence}, this shows that for weighted
Sobolev spaces $H_{x}^{s_{x}}H_{v}^{s_{v},b}$, the critical $v-$regularity for
the existence of nontrivial invariant structures near a stable homogeneous
equilibrium is $s_{v}=\frac{3}{2}$. This gives a generalization of the $1$D
results in \cite{lin-zeng-bgk} to higher dimensions. We note that the critical
regularity $s_{v}=\frac{3}{2}$ does not depend on the dimension. This
illustrates again the longitudinal ($1$D) nature of Landau damping, which is
obvious in the linear regime.

We briefly mention some differences of the long time behaviors of
Vlasov-Poisson in $1$D and higher dimensions. For the $1$D case, numerical
simulations (e.g. \cite{deimo-zweifel} \cite{driscoll-et-04}) indicated that
for certain small initial data near a stable homogeneous state including
Maxwellian, there is no decay of electric fields and the asymptotic state is a
BGK wave or superposition of BGK waves. Moreover, BGK waves also appear as the
asymptotic states for the saturation of an unstable homogeneous state
(\cite{amstrong-monto-67}) in $1$D. These suggest that small BGK waves play
important role in understanding the long time behaviors of $1$D Vlasov-Poisson
system. However, for $2$ and $3$D Vlasov-Poisson, numerical simulations
(\cite{morse-nelson69} \cite{oppen-et-numerical-3d}) suggested that when
starting near a homogeneous state, the electric fields decay eventually. Our
Theorems \ref{thm-existence} and \ref{Thm-existence-weighted} on existence of
$1$D BGK waves show that such decay of electric field is not true for general
initial data near homogeneous states. But the numerical simulations seem to
suggest that these $1$D BGK waves do not appear in the long time dynamics in
$2$D and$\ 3$D. To explain these phenomena, it will be interesting to
understand the transversal instability of 1D BGK waves.

This paper is organized as follows. In Section 2, we prove the existence of
$1$D BGK waves in $W^{s,p}\ \left(  s<1+\frac{1}{p}\right)  \ $neighborhoods
of homogeneous states. In Section 3, we use the linear decay estimate to show
that all invariant structures near stable homogeneous equilibria in
$H_{x}^{s_{x}}H_{v}^{s_{v},b}$ spaces satisfying
(\ref{condition-non-existence}) are trivial. Throughout this paper, we use $C$
to denote a generic constant in the estimates and the dependence of $C$ is
indicated only when it matters in the proof.

\section{Existence of BGK waves in $W^{s,p}\ \left(  s<1+\frac{1}{p}\right)
$}

In this Section, we construct nontrivial steady states (BGK waves) near any
homogeneous state in the space $W_{x,v}^{s,p}$ $\left(  s<1+\frac{1}%
{p}\right)  $. We consider $d=2$ only, since the proof is almost the same for
$d=3.\ $The BGK waves we construct are one-dimensional, that is, the steady
distribution $f=f\left(  x_{1},v_{1},v_{2}\right)  $ and the electric field
$\vec{E}=E\left(  x_{1}\right)  \vec{e}_{1}.$ We will show that such a
restriction is necessary by excluding $2D$ and $3D\ $BGK waves.

\begin{proof}
[Proof of Theorem \ref{thm-existence}]We adapt the line of proof in
\cite{lin-zeng-bgk} to construct BGK wave solutions for $2D$ Vlasov-Poisson
equations. First, we modify $f_{0}\left(  v\right)  $ to a smooth function
$f_{1}\left(  v\right)  $ with some additional properties. In the first step,
let $\eta\left(  v\right)  ~(v\in\mathbf{R}^{2})\ $be the standard mollifier
function. For $\delta_{1}>0$ define $f_{\delta_{1}}\left(  v\right)
=\eta_{\delta_{1}}\left(  v\right)  \ast f_{0}\left(  v\right)  $, where
$\eta_{\delta_{1}}\left(  v\right)  =\frac{1}{\delta_{1}^{2}}\eta\left(
\frac{v}{\delta_{1}}\right)  $. Then by the properties of mollifiers, we have%
\[
f_{\delta_{1}}\in C^{\infty}\left(  \mathbf{R}\right)  ,\ f_{\delta_{1}%
}\left(  v\right)  \geq0,\ \int_{\mathbf{R}^{2}}f_{\delta_{1}}\left(
v\right)  dv=1,
\]
and when $\delta_{1}$ is small enough
\[
\left\Vert f_{\delta_{1}}-f_{0}\right\Vert _{L^{1}\left(  \mathbf{R}%
^{2}\right)  }+\int_{\mathbf{R}^{2}}\left\vert v\right\vert ^{2}\left\vert
f_{\delta_{1}}-f_{0}\right\vert \ dv+\left\Vert f_{\delta_{1}}-f_{0}%
\right\Vert _{W^{s,p}\left(  \mathbf{R}^{2}\right)  }\leq\frac{\varepsilon}%
{6}.
\]
Modifying $f_{\delta_{1}}\left(  v\right)  $ near infinity by cut-off, we can
assume in addition that $f_{\delta_{1}}\left(  v\right)  \in H^{2,b}\left(
\mathbf{R}^{2}\right)  $ (defined in (\ref{definition-H-tilde})). In the
second step, let $\sigma\left(  x_{1}\right)  =\sigma\left(  \left\vert
x_{1}\right\vert \right)  $ be the $1$D\ cut-off function. Let $\delta_{2}>0$
be a small number, and define%
\begin{align*}
f_{\delta_{1},\delta_{2}}\left(  v_{1},v_{2}\right)   &  =f_{\delta_{1}%
}\left(  v_{1},v_{2}\right)  \left(  1-\sigma\left(  \frac{v_{1}}{\delta_{2}%
}\right)  \right)  +\left(  \frac{f_{\delta_{1}}\left(  v_{1},v_{2}\right)
+f_{\delta_{1}}\left(  -v_{1},v_{2}\right)  }{2}\right)  \sigma\left(
\frac{v_{1}}{\delta_{2}}\right) \\
&  =f_{\delta_{1}}\left(  v_{1},v_{2}\right)  -\left(  \frac{f_{\delta_{1}%
}\left(  v_{1},v_{2}\right)  -f_{\delta_{1}}\left(  -v_{1},v_{2}\right)  }%
{2}\right)  \sigma\left(  \frac{v_{1}}{\delta_{2}}\right)  .
\end{align*}
Then,
\[
f_{\delta_{1},\delta_{2}}\in C^{\infty}\left(  \mathbf{R}^{2}\right)
,\ f_{\delta_{1},\delta_{2}}\left(  v\right)  >0,\ \int_{\mathbf{R}^{2}%
}f_{\delta_{1},\delta_{2}}\left(  v\right)  dv=\int_{\mathbf{R}^{2}}%
f_{\delta_{1}}\left(  v\right)  dv=1,
\]
and $f_{\delta_{1},\delta_{2}}\left(  v_{1},v_{2}\right)  $ is even in $v_{1}$
when $v_{1}\in\left[  -\delta_{2},\delta_{2}\right]  $. We show that: when
$\delta_{2}$ is small enough
\begin{equation}
\left\Vert f_{\delta_{1},\delta_{2}}-f_{\delta_{1}}\right\Vert _{L^{1}\left(
\mathbf{R}^{2}\right)  }+\int_{\mathbf{R}^{2}}\left\vert v\right\vert
^{2}\left\vert f_{\delta_{1},\delta_{2}}-f_{\delta_{1}}\right\vert
\ dv+\left\Vert f_{\delta_{1},\delta_{2}}-f_{\delta_{1}}\right\Vert
_{W^{s,p}\left(  \mathbf{R}^{2}\right)  }\leq\frac{\varepsilon}{6}.
\label{estimate-delta-2}%
\end{equation}
A minor modification of the proof of Lemma 2.2 in \cite{lin-zeng-bgk} yields
that: when $\delta_{2}\rightarrow0,$%
\[
\left\Vert f_{\delta_{1}}-f_{\delta_{1},\delta_{2}}\right\Vert _{L^{1}\left(
\mathbf{R}^{2}\right)  }+\int_{\mathbf{R}^{2}}\left\vert v\right\vert
^{2}\left\vert f_{\delta_{1}}-f_{\delta_{1},\delta_{2}}\right\vert
\ dv+\ \left\Vert f_{\delta_{1}}-f_{\delta_{1},\delta_{2}}\right\Vert
_{W^{1,p}\left(  \mathbf{R}^{2}\right)  }\rightarrow0.
\]
It remains to show that%
\begin{equation}
\left\Vert \nabla\left(  f_{\delta_{1}}-f_{\delta_{1},\delta_{2}}\right)
\right\Vert _{W^{s-1,p}\left(  \mathbf{R}^{2}\right)  }\rightarrow0\text{,
when }\delta_{2}\rightarrow0\text{.} \label{estimate-gradient-delta-2}%
\end{equation}
We have
\[
\partial_{v_{2}}\left(  f_{\delta_{1}}-f_{\delta_{1},\delta_{2}}\right)
=\left(  \frac{\partial_{v_{2}}f_{\delta_{1}}\left(  v_{1},v_{2}\right)
-\partial_{v_{2}}f_{\delta_{1}}\left(  -v_{1},v_{2}\right)  }{2}\right)
\sigma\left(  \frac{v_{1}}{\delta_{2}}\right)  ,
\]
and%
\begin{align*}
\partial_{v_{1}}\left(  f_{\delta_{1}}-f_{\delta_{1},\delta_{2}}\right)   &
=\left(  \frac{\partial_{v_{1}}f_{\delta_{1}}\left(  v_{1},v_{2}\right)
+\partial_{v_{1}}f_{\delta_{1}}\left(  -v_{1},v_{2}\right)  }{2}\right)
\sigma\left(  \frac{v_{1}}{\delta_{2}}\right) \\
&  +\sigma^{\prime}\left(  \frac{v_{1}}{\delta_{2}}\right)  \frac{v_{1}%
}{\delta_{2}}\frac{f_{\delta_{1}}\left(  v_{1},v_{2}\right)  -f_{\delta_{1}%
}\left(  -v_{1},v_{2}\right)  }{2v_{1}}.
\end{align*}
By a scaling argument as in the proof of Lemma 2.2 of \cite{lin-zeng-bgk},
\[
\left\Vert \frac{f_{\delta_{1}}\left(  v_{1},v_{2}\right)  -f_{\delta_{1}%
}\left(  -v_{1},v_{2}\right)  }{2v_{1}}\right\Vert _{W^{s-1,p}\left(
\mathbf{R}^{2}\right)  }\leq C\left\Vert f_{\delta_{1}}\right\Vert
_{W^{s,p}\left(  \mathbf{R}^{2}\right)  }.
\]
So (\ref{estimate-gradient-delta-2}) follows from Lemma \ref{lemma-delta-0}
below. Thus for fixed $\varepsilon>0,\ $by choosing $\delta_{1},\delta_{2}$
small enough, we get
\[
\left\Vert f_{\delta_{1},\delta_{2}}-f_{0}\right\Vert _{L^{1}\left(
\mathbf{R}^{2}\right)  }+\int_{\mathbf{R}^{2}}\left\vert v\right\vert
^{2}\left\vert f_{\delta_{1},\delta_{2}}-f_{0}\right\vert \ dv+\left\Vert
f_{\delta_{1},\delta_{2}}-f_{0}\right\Vert _{W^{s,p}\left(  \mathbf{R}%
^{2}\right)  }\leq\frac{\varepsilon}{3}.
\]
We set $f_{1}\left(  v_{1},v_{2}\right)  =f_{\delta_{1},\delta_{2}}\left(
v_{1},v_{2}\right)  $, then
\[
f_{1}\left(  v\right)  >0,\ \ f_{1}\left(  v\right)  \in C^{\infty}\left(
\mathbf{R}^{2}\right)  \cap\tilde{H}^{2}\left(  \mathbf{R}^{2}\right)
,\ \int_{\mathbf{R}^{2}}f_{1}\left(  v\right)  dv=1,
\]
$f_{1}\left(  v\right)  $ is even for $v_{1}$ in $\left[  -\delta_{2}%
,\delta_{2}\right]  $ and within $\frac{\varepsilon}{3}$ distance of
$f_{0}\left(  v\right)  $ in the norm of (\ref{norm-thm-stability}). Below, we
denote $a=\delta_{2}/2.$

Fix the $x_{1}-$period $T_{1}>0$, we only consider the travel speed $c=0$
since the construction for any $c\in\mathbf{R}$ follows by the Galilean
transform as in \cite{lin-zeng-bgk}. Our strategy is to construct BGK wave
solutions of the form $\left(  f_{\varepsilon}\left(  x_{1},v_{1}%
,v_{2}\right)  ,E_{\varepsilon}\left(  x_{1}\right)  \vec{e}_{1}\right)  \ $by
bifurcation at a modified homogeneous profile near $f_{1}\left(  v_{1}%
,v_{2}\right)  $. Denote $\sigma\left(  x\right)  =\sigma\left(  \left\vert
x\right\vert \right)  $ to be the cut-off function such that$\ \sigma\left(
x\right)  \in C_{0}^{\infty}\left(  \mathbf{R}\right)  ,$
\begin{equation}
\ 0\leq\sigma\left(  x\right)  \leq1;\ \sigma\left(  x\right)  =1\text{ when
}\left\vert x\right\vert \leq1\text{; }\sigma\left(  x\right)  =0\text{ when
}\left\vert x\right\vert \geq2\text{.} \label{cut-off}%
\end{equation}
Similar to Lemma 2.1 in \cite{lin-zeng-bgk}, there exists $g_{0}\left(
x_{1},x_{2}\right)  \in C^{\infty}\left(  \mathbf{R}^{2}\right)  ,\ g_{0}=0$
when $\left\vert x_{1}\right\vert \geq4a^{2},$ such that
\[
f_{1}\left(  v_{1},v_{2}\right)  \sigma\left(  \frac{v_{1}}{a}\right)
=g_{0}\left(  v_{1}^{2},v_{2}\right)  .
\]
Define $g_{+}\left(  x_{1},x_{2}\right)  ,\ g_{-}\left(  x_{1},x_{2}\right)
\in C^{\infty}\left(  \mathbf{R}^{2}\right)  $ by
\[
g_{\pm}\left(  x_{1},x_{2}\right)  =\left\{
\begin{array}
[c]{cc}%
f_{1}\left(  \pm\sqrt{x_{1}},x_{2}\right)  \left(  1-\sigma\left(  \frac
{\sqrt{x_{1}}}{a}\right)  \right)  +g_{0}\left(  x_{1},x_{2}\right)  &
\text{if }x_{1}>a^{2}\\
g_{0}\left(  x_{1},x_{2}\right)  & \text{if }-4a^{2}<x_{1}\leq a^{2}\\
0 & \text{if }x_{1}\leq-4a^{2}%
\end{array}
.\right.
\]
Then%
\[
f_{1}\left(  v_{1},v_{2}\right)  =\left\{
\begin{array}
[c]{cc}%
g_{+}\left(  v_{1}^{2},v_{2}\right)  & \text{if }v_{1}>0\\
g_{-}\left(  v_{1}^{2},v_{2}\right)  & \text{if }v_{1}\leq0
\end{array}
\right.  .
\]
Since $\partial_{v_{1}}f_{1}\left(  0,v_{2}\right)  =0$, $\ f_{1}\in
C^{\infty}\left(  \mathbf{R}^{2}\right)  \cap H^{2,b}\left(  \mathbf{R}%
^{2}\right)  $, we have
\[
\left\vert \int_{\mathbf{R}^{2}}\frac{\partial_{v_{1}}f_{1}\left(  v\right)
}{v_{1}}dv\right\vert <\infty.
\]
Indeed, let $\bar{f}_{1}\left(  v_{1}\right)  =\int_{\mathbf{R}}f_{1}\left(
v_{1},v_{2}\right)  \ dv_{2}$, then since $\bar{f}_{1}^{\prime}\left(
0\right)  =0$,\ by Corollary \ref{cor-inequality},
\[
\left\vert \int_{\mathbf{R}^{2}}\frac{\partial_{v_{1}}f_{1}\left(  v\right)
}{v_{1}}dv\right\vert =\left\vert \int_{\mathbf{R}}\frac{\bar{f}_{1}^{\prime
}\left(  v_{1}\right)  }{v_{1}}dv_{1}\right\vert \leq C\left\Vert
f_{1}\right\Vert _{H^{2,b}\left(  \mathbf{R}^{2}\right)  }\text{ .}%
\]
We consider three cases.

Case 1: $\int_{\mathbf{R}^{2}}\frac{\partial_{v_{1}}f_{1}\left(  v\right)
}{v_{1}}dv<\left(  \frac{2\pi}{T_{1}}\right)  ^{2}.\ \ $Let
\[
F_{1}\left(  v_{1}\right)  =\exp\left(  -\frac{\left(  v_{1}-v_{0}\right)
^{2}}{2}\right)  +\exp\left(  -\frac{\left(  v_{1}+v_{0}\right)  ^{2}}%
{2}\right)  =G_{1}\left(  v_{1}^{2}\right)  ,~\
\]
and $F_{2}\left(  v_{2}\right)  =e^{-\frac{1}{2}v_{2}^{2}},\ $where $v_{0}$ is
a large positive constant such that
\[
\int_{\mathbf{R}}\frac{F_{1}^{\prime}\left(  v_{1}\right)  }{v_{1}}dv_{1}>0.
\]
Let $\gamma,\delta>0$ be two small parameters to be fixed, define
\begin{equation}
f_{\gamma,\delta}\left(  v_{1},v_{2}\right)  =\frac{1}{1+C_{0}\gamma^{2}%
}\left[  f_{1}\left(  v_{1},v_{2}\right)  +\frac{\gamma}{\delta}F_{1}\left(
\frac{v_{1}}{\gamma\delta}\right)  F_{2}\left(  v_{2}\right)  \right]  ,
\label{defn-f-gamma-delta}%
\end{equation}
where $C_{0}=\int F_{1}\left(  v_{1}\right)  F_{2}\left(  v_{2}\right)  dv>0$.
The rest of the proof is similar to the proof of Proposition 2.1 in
\cite{lin-zeng-bgk}. We sketch it below. There exists $0<\delta_{1}<\delta
_{2}$ such that for $\gamma_{0}>0$ small enough
\begin{equation}
0<\int_{\mathbf{R}^{2}}\frac{\partial_{v_{1}}f_{\gamma,\delta_{2}}\left(
v_{1},v_{2}\right)  }{v_{1}}dv<\left(  \frac{2\pi}{T_{1}}\right)  ^{2}%
<\int_{\mathbf{R}^{2}}\frac{\partial_{v_{1}}f_{\gamma,\delta_{1}}\left(
v_{1},v_{2}\right)  }{v_{1}}dv,\text{ when }0<\gamma<\gamma_{0}\text{.}
\label{inequality-period}%
\end{equation}
Let $\beta\left(  x_{1}\right)  $ be a $T_{1}$ periodic function and denote
$e=\frac{1}{2}v_{1}^{2}-\beta\left(  x_{1}\right)  $.\ We look for $1$D BGK
wave solution
\[
f^{0}=f_{\gamma,\delta}^{\beta}\left(  x_{1},v\right)  ,\ \ \vec{E}^{0}%
=E^{0}\left(  x_{1}\right)  \vec{e}_{1}%
\]
near $\left(  f_{\gamma,\delta},0\right)  $, where
\begin{equation}
f_{\gamma,\delta}^{\beta}\left(  x_{1},v\right)  =\left\{
\begin{array}
[c]{cc}%
\frac{1}{1+C_{0}\gamma^{2}}\left[  g_{+}\left(  2e,v_{2}\right)  +\frac
{\gamma}{\delta}G_{1}\left(  \frac{2e}{\left(  \gamma\delta\right)  ^{2}%
}\right)  F_{2}\left(  v_{2}\right)  \right]  & \text{if }v_{1}>0\\
\frac{1}{1+C_{0}\gamma^{2}}\left[  g_{-}\left(  2e,v_{2}\right)  +\frac
{\gamma}{\delta}G_{1}\left(  \frac{2e}{\left(  \gamma\delta\right)  ^{2}%
}\right)  F_{2}\left(  v_{2}\right)  \right]  & \text{if }v_{1}\leq0
\end{array}
\right.  \label{defn-f-steady}%
\end{equation}
and $E^{0}\left(  x_{1}\right)  =-\beta^{\prime}\left(  x_{1}\right)  $. The
steady Vlasov-Poisson equation is reduced to the ODE
\begin{align}
\beta^{\prime\prime}  &  =\int_{\mathbf{R}^{2}}f_{\gamma,\delta}^{\beta
}\left(  x,v\right)  \ dv-1\label{ode-beta}\\
&  =\frac{1}{1+C_{0}\gamma^{2}}\left[  \int_{v_{1}>0}g_{+}\left(
2e,v_{2}\right)  dv+\int_{v_{1}\leq0}g_{-}\left(  2e,v_{2}\right)
dv+\int_{\mathbf{R}^{2}}\frac{\gamma}{\delta}G_{1}\left(  \frac{2e}{\left(
\gamma\delta\right)  ^{2}}\right)  F_{2}\left(  v_{2}\right)  dv\right]
-1\nonumber\\
&  :=h_{\gamma,\delta}\left(  \beta\right)  .\nonumber
\end{align}
Since
\[
h_{\gamma,\delta}\left(  0\right)  =\int_{\mathbf{R}^{d}}f_{\gamma,\delta
}\left(  v\right)  \ dv-1=0
\]
and
\begin{align*}
&  \ \ \ \ \ \ h_{\gamma,\delta}^{\prime}\left(  0\right) \\
&  =\frac{-2}{1+C_{0}\gamma^{2}}\left\{  \int_{v_{1}>0}\partial_{1}%
g_{+}\left(  v_{1}^{2},v_{2}\right)  dv+\int_{v_{1}\leq0}\partial_{1}%
g_{-}\left(  v_{1}^{2},v_{2}\right)  dv+\int_{\mathbf{R}^{2}}\frac{\gamma
}{\delta}\frac{1}{\left(  \gamma\delta\right)  ^{2}}G^{\prime}\left(
\frac{v_{1}^{2}}{\left(  \gamma\delta\right)  ^{2}}\right)  F_{2}\left(
v_{2}\right)  dv\right\} \\
&  =-\int_{\mathbf{R}^{2}}\frac{\partial_{v_{1}}f_{\gamma,\delta}\left(
v_{1},v_{2}\right)  }{v_{1}}dv<0,\text{ when }0<\gamma<\gamma_{0},\ \delta
_{1}<\delta<\delta_{2},
\end{align*}
so $\beta=0$ is a center for the ODE (\ref{ode-beta}) and there exist
bifurcation of periodic solutions. More precisely, for any fixed $\gamma
\in\left(  0,\gamma_{0}\right)  ,\ $there exists $r_{0}>0$ (independent of
$\delta\in\left(  \delta_{1},\delta_{2}\right)  $)$\,$, such that for each
$0<r<r_{0}\,$, there exists a $T\left(  \gamma,\delta;r\right)  -$periodic
solution $\beta_{\gamma,\delta;r}$ to the ODE (\ref{ode-beta}) with
$\left\Vert \beta_{\gamma,\delta;r}\right\Vert _{H^{2}\left(  0,T\left(
\gamma,\delta;r\right)  \right)  }=r$. Moreover,
\[
\left(  \frac{2\pi}{T\left(  \gamma,\delta;r\right)  }\right)  ^{2}%
\rightarrow\int_{\mathbf{R}^{2}}\frac{\partial_{v_{1}}f_{\gamma,\delta}\left(
v_{1},v_{2}\right)  }{v_{1}}dv\text{, when }r\rightarrow0.
\]
To get a solution with the given period $T_{1},$ we adjust $\delta\in\left[
\delta_{1},\delta_{2}\right]  \ \ $by using the inequality
(\ref{inequality-period}) and the fact that $T\left(  \gamma,\delta;r\right)
$ is continuous to $\delta$. So for each $\gamma,r>0$ small enough, there
exists $\delta_{T_{1}}\left(  \gamma,r\right)  \in\left(  \delta_{1}%
,\delta_{2}\right)  \,$, such that $T\left(  \gamma,\delta_{T_{1}};r\right)
=T_{1}$. Define $f_{\gamma,r}\left(  x_{1},v\right)  =f_{\gamma,\delta_{T1}%
}^{\beta}\left(  x,v\right)  $,$\ \beta_{\gamma,r}\left(  x_{1}\right)
=\beta_{\gamma,\delta_{T1};r}$ and let $\vec{E}_{\gamma,r}\left(
x_{1}\right)  =-\beta_{\gamma,r}^{\prime}\left(  x_{1}\right)  \vec{e}_{1}$.
Then $\left(  f_{\gamma,r}\left(  x_{1},v\right)  ,\vec{E}_{\gamma,r}\left(
x_{1}\right)  \right)  $ is a nontrivial steady solution to (\ref{vpe}) with
$x_{1}-$period $T_{1}$. For any fixed $\gamma>0$, let
\[
\delta\left(  \gamma\right)  =\lim_{r\rightarrow0}\delta_{T_{1}}\left(
\gamma,r\right)  \in\left[  \delta_{1},\delta_{2}\right]  .
\]
By the dominant convergence theorem, it is easy to show that
\[
\left\Vert f_{\gamma,r}\left(  x_{1},v\right)  -f_{\gamma,\delta\left(
\gamma\right)  }\left(  v\right)  \right\Vert _{L_{x_{1},v}^{1}}+\ \int
_{0}^{T_{1}}\int_{\mathbf{R}^{2}}\left\vert v\right\vert ^{2}\left\vert
f_{\gamma,r}\left(  x_{1},v\right)  -f_{\gamma,\delta\left(  \gamma\right)
}\left(  v\right)  \right\vert \ dx_{1}dv\ \ \ \
\]%
\[
\ +\left\Vert f_{\gamma,r}\left(  x,v\right)  -f_{\gamma,\delta\left(
\gamma\right)  }\left(  v\right)  \right\Vert _{W_{x_{1},v}^{2,p}}%
\rightarrow0,\ \ \ \
\]
when $r=\left\Vert \beta_{\gamma,r}\left(  x_{1}\right)  \right\Vert
_{H^{2}\left(  0,T_{1}\right)  }\rightarrow0.\ $So for any $\gamma>0$ and
$\varepsilon>0$, there exists $r=r\left(  \gamma,\varepsilon\right)  >0$ such
that
\[
\left\Vert f_{\gamma,r}\left(  x_{1},v\right)  -f_{\gamma,\delta\left(
\gamma\right)  }\left(  v\right)  \right\Vert _{L_{x_{1},v}^{1}}+\ \int
_{0}^{T_{1}}\int_{\mathbf{R}^{2}}\left\vert v\right\vert ^{2}\left\vert
f_{\gamma,r}\left(  x_{1},v\right)  -f_{\gamma,\delta\left(  \gamma\right)
}\left(  v\right)  \right\vert \ dx_{1}dv\ \
\]%
\[
\ +\left\Vert f_{\gamma,r}\left(  x,v\right)  -f_{\gamma,\delta\left(
\gamma\right)  }\left(  v\right)  \right\Vert _{W_{x_{1},v}^{2,p}}%
<\frac{\varepsilon}{3}.
\]
Since%
\[
f_{1}\left(  v\right)  -f_{\gamma,\delta\left(  \gamma\right)  }\left(
v\right)  =\frac{1}{1+C_{0}\gamma^{2}}\left[  -C_{0}\gamma^{2}f_{1}\left(
v\right)  -\frac{\gamma}{\delta}F_{1}\left(  \frac{v_{1}}{\gamma\delta
}\right)  F_{2}\left(  v_{2}\right)  \right]  .
\]
and $\delta\left(  \gamma\right)  \in\left[  \delta_{1},\delta_{2}\right]
$,\ by using Lemma \ref{lemma-delta-0}, for $s<1+\frac{1}{p},$
\[
\left\Vert f_{1}\left(  v\right)  -f_{\gamma,\delta\left(  \gamma\right)
}\left(  v\right)  \right\Vert _{W^{s,p}\left(  \mathbf{R}^{2}\right)
}\rightarrow0,\ \text{\ when }\gamma\rightarrow0.
\]
It is also easy to show that
\[
\left\Vert f_{1}\left(  v\right)  -f_{\gamma,\delta\left(  \gamma\right)
}\left(  v\right)  \right\Vert _{L^{1}}+\ \int_{\mathbf{R}^{2}}\left\vert
v\right\vert ^{2}\left\vert f_{1}\left(  v\right)  -f_{\gamma,\delta\left(
\gamma\right)  }\left(  v\right)  \right\vert \ dv\rightarrow0,\ \ \text{when
}\gamma\rightarrow0.
\]
Thus we can choose $\gamma>0$ small enough such that
\[
T_{1}\left\Vert f_{1}\left(  v\right)  -f_{\gamma,\delta\left(  \gamma\right)
}\left(  v\right)  \right\Vert _{L^{1}}+\ T_{1}\int_{\mathbf{R}^{2}}\left\vert
v\right\vert ^{2}\left\vert f_{1}\left(  v\right)  -f_{\gamma,\delta\left(
\gamma\right)  }\left(  v\right)  \right\vert dv\ \
\]%
\[
\ +\left\Vert f_{1}\left(  v\right)  -f_{\gamma,\delta\left(  \gamma\right)
}\left(  v\right)  \right\Vert _{W^{s,p}\left(  \mathbf{R}^{2}\right)  }%
<\frac{\varepsilon}{3}.
\]
So the nontrivial steady solution $\left(  f_{\gamma,r}\left(  x_{1},v\right)
,\vec{E}_{\gamma,r}\left(  x_{1}\right)  \right)  $ is within $\varepsilon$
distance of the homogeneous state $\left(  f_{0}\left(  v\right)  ,\vec
{0}\right)  $ in the norm of (\ref{norm-thm-stability}).

Case 2: $\int_{\mathbf{R}^{2}}\frac{\partial_{v_{1}}f_{1}\left(  v\right)
}{v_{1}}dv<\left(  \frac{2\pi}{T_{1}}\right)  ^{2}$. Choose $F_{1}\left(
v_{1}\right)  =\exp\left(  -\frac{v_{1}^{2}}{2}\right)  $ and $F_{2}\left(
v_{2}\right)  $ is the same as before. Define $f_{\gamma,\delta}\left(
v\right)  $ as in Case 1 (see (\ref{defn-f-gamma-delta})). Then there exists
$0<\delta_{1}<\delta_{2}$ such that%
\[
0<\int_{\mathbf{R}^{2}}\frac{\partial_{v_{1}}f_{\gamma,\delta_{1}}\left(
v_{1},v_{2}\right)  }{v_{1}}dv<\left(  \frac{2\pi}{T_{1}}\right)  ^{2}%
<\int_{\mathbf{R}^{2}}\frac{\partial_{v_{1}}f_{\gamma,\delta_{2}}\left(
v_{1},v_{2}\right)  }{v_{1}}dv.
\]
The rest of the proof is the same as in Case 1.

Case 3: $\int_{\mathbf{R}^{2}}\frac{\partial_{v_{1}}f_{1}\left(  v\right)
}{v_{1}}dv=\left(  \frac{2\pi}{T_{1}}\right)  ^{2}$. For $\delta>0,\ $define
$f_{\delta}\left(  v_{1},v_{2}\right)  =\frac{1}{\delta}f_{1}\left(
\frac{v_{1}}{\delta},v_{2}\right)  .$ For any $\varepsilon>0$ , there exist
$0<\delta_{1}\left(  \varepsilon\right)  <1<\delta_{2}\left(  \varepsilon
\right)  $ such that
\[
0<\int_{\mathbf{R}^{2}}\frac{\partial_{v_{1}}f_{\delta_{2}}\left(  v\right)
}{v_{1}}dv<\left(  \frac{2\pi}{T_{1}}\right)  ^{2}<\int_{\mathbf{R}^{2}}%
\frac{\partial_{v_{1}}f_{\delta_{1}}\left(  v\right)  }{v}dv,
\]
and when $\delta\in\left(  \delta_{1}\left(  \varepsilon\right)  ,\delta
_{2}\left(  \varepsilon\right)  \right)  ,$
\[
T_{1}\left\Vert f_{1}\left(  v\right)  -f_{\delta}\left(  v\right)
\right\Vert _{L^{1}\left(  \mathbf{R}^{2}\right)  }+\ T_{1}\int_{\mathbf{R}%
^{2}}\left\vert v\right\vert ^{2}\left\vert f_{1}\left(  v\right)  -f_{\delta
}\left(  v\right)  \right\vert dv\ \
\]%
\[
\ +\left\Vert f_{1}\left(  v\right)  -f_{\delta}\left(  v\right)  \right\Vert
_{W^{s,p}\left(  \mathbf{R}^{2}\right)  }<\frac{\varepsilon}{3}.
\]
$\ $We construct steady BGK waves near $\left(  f_{\delta}\left(  v\right)
,\vec{0}\right)  $, which are of the form
\begin{equation}
f_{\delta}^{\beta}\left(  x_{1},v\right)  =\left\{
\begin{array}
[c]{cc}%
\frac{1}{\delta}g_{+}\left(  \frac{2e}{\delta^{2}},v_{2}\right)   & \text{if
}v_{1}>0\\
\frac{1}{\delta}g_{-}\left(  \frac{2e}{\delta^{2}},v_{2}\right)   & \text{if
}v_{1}\leq0
\end{array}
\right.  ,\text{ }e=\frac{1}{2}v_{1}^{2}-\beta\left(  x_{1}\right)
,\label{defn-f-steady-case3}%
\end{equation}
and $\vec{E}^{0}=-\beta^{\prime}\left(  x_{1}\right)  \vec{e}_{1}$. The
existence of BGK waves is then reduced to solve the ODE
\begin{equation}
\beta^{\prime\prime}=\int_{\mathbf{R}^{2}}f_{\delta}^{\beta}\left(
x_{1},v\right)  \ dv-1:=h_{\delta}\left(  \beta\right)
.\label{ode-beta-case3}%
\end{equation}
As in Case 1, for any $\delta\in\left(  \delta_{1}\left(  \varepsilon\right)
,\delta_{2}\left(  \varepsilon\right)  \right)  ,\ \exists\ r_{0}\left(
\varepsilon\right)  >0$ (independent of $\delta$) such that for each
$0<r<r_{0}\,$, there exists a $T\left(  \delta;r\right)  -$periodic solution
$\beta_{\delta;r}$ to the ODE $(\ref{ode-beta-case3})$, satisfying $\left\Vert
\beta_{\delta;r}\right\Vert _{H^{2}\left(  0,T\left(  \delta;r\right)
\right)  }=r$ and
\[
\left(  \frac{2\pi}{T\left(  \delta;r\right)  }\right)  ^{2}\rightarrow
\int_{\mathbf{R}^{2}}\frac{\partial_{v_{1}}f_{\delta}\left(  v\right)  }%
{v_{1}}dv\text{, when }r\rightarrow0.
\]
For $r$ small enough, again there exists $\delta_{T_{1}}\left(  r,\varepsilon
\right)  \in\left(  \delta_{1}\left(  \varepsilon\right)  ,\delta_{2}\left(
\varepsilon\right)  \right)  $ such that $T\left(  \delta_{T_{1}};r\right)
=T_{1}$. Define $f_{r,\varepsilon}\left(  x_{1},v\right)  =f_{\delta_{T_{1}}%
}^{\beta}\left(  x_{1},v\right)  $ and $\vec{E}_{r,\varepsilon}\left(
x\right)  =-\beta_{\delta_{T1};r}^{\prime}\left(  x_{1}\right)  \vec{e}_{1}$.
Then $\left(  f_{r,\varepsilon}\left(  x_{1},v\right)  ,\vec{E}_{r,\varepsilon
}\left(  x\right)  \right)  $ is a nontrivial steady solution to ($\ref{vpe})$
with $x_{1}-$period $T_{1}$. As in Cases 1 and 2, by choosing $r$ small
enough, $f_{r,\varepsilon}\left(  x_{1},v\right)  $ is within $\varepsilon$
distance of the homogeneous state $\left(  f_{0}\left(  v\right)  ,0\right)  $
in the norm of (\ref{norm-thm-stability}). This finishes the proof of the
Theorem \ref{thm-existence}.
\end{proof}

\begin{lemma}
\label{lemma-delta-0}

(i) Given $f\in W^{\frac{1}{p},p}\left(  \mathbf{R}\right)  \cap L^{\infty
}\left(  \mathbf{R}\right)  ,$and$\ $%
\[
g\in W^{s,p}\left(  \mathbf{R}^{2}\right)  \ \left(  p>1,0\leq s<\frac{1}%
{p}\right)  .
\]
Then for $\delta>0,\ $
\begin{equation}
\left\Vert f\left(  \frac{v_{1}}{\delta}\right)  g\left(  v_{1},v_{2}\right)
\right\Vert _{W^{s,p}\left(  \mathbf{R}^{2}\right)  }\rightarrow0\text{, when
}\delta\rightarrow0\text{. } \label{zero-lemma}%
\end{equation}

(ii) Given $f,g\in W^{s,p}\left(  \mathbf{R}\right)  $ $\left(  p>1,0\leq
s<\frac{1}{p}\right)  $.~Then for $\delta>0,\ $
\[
\left\Vert f\left(  \frac{v_{1}}{\delta}\right)  g\left(  v_{2}\right)
\right\Vert _{W^{s,p}\left(  \mathbf{R}^{2}\right)  }\rightarrow0\text{, when
}\delta\rightarrow0\text{. }%
\]

\end{lemma}

\begin{proof}
Proof of (i): First we consider $g\in C_{0}^{\infty}\left(  \mathbf{R}%
^{2}\right)  $. By Fubini Theorem for $W^{s,p}\left(  \mathbf{R}^{2}\right)  $
norm (see \cite{strichartz}), we have
\begin{align*}
&  \left\Vert f\left(  \frac{v_{1}}{\delta}\right)  g\left(  v_{1}%
,v_{2}\right)  \right\Vert _{W^{s,p}\left(  \mathbf{R}^{2}\right)  }\\
&  \leq C\left(  \left\Vert \left\Vert f\left(  \frac{v_{1}}{\delta}\right)
g\left(  v_{1},v_{2}\right)  \right\Vert _{W_{v_{1}}^{s,p}\left(
\mathbf{R}\right)  }\right\Vert _{L_{v_{2}}^{p}}+\left\Vert f\left(
\frac{v_{1}}{\delta}\right)  \left\Vert g\left(  v_{1},v_{2}\right)
\right\Vert _{W_{v_{2}}^{s,p}\left(  \mathbf{R}\right)  }\right\Vert
_{L_{v_{1}}^{p}}\right)  .
\end{align*}
By the estimates in the proof of Theorem 3.2 of \cite{strichartz}, for any
$p>1,\ s<\frac{1}{p},$ when $h_{1}\in W^{s,p}\left(  \mathbf{R}\right)
,\ h_{2}\in W^{\frac{1}{p},p}\left(  \mathbf{R}\right)  \cap L^{\infty}\left(
\mathbf{R}\right)  ,$ we have%
\[
\left\Vert h_{1}h_{2}\right\Vert _{W^{s,p}}\leq C\left\Vert h_{1}\right\Vert
_{W^{s,p}}\left(  \left\Vert h_{2}\right\Vert _{W^{\frac{1}{p},p}}+\left\Vert
h_{2}\right\Vert _{L^{\infty}}\right)  .
\]
So
\begin{align*}
&  \left\Vert \left\Vert f\left(  \frac{v_{1}}{\delta}\right)  g\left(
v_{1},v_{2}\right)  \right\Vert _{W_{v_{1}}^{s,p}\left(  \mathbf{R}\right)
}\right\Vert _{L_{v_{2}}^{p}}\\
&  \leq C\left\Vert f\left(  \frac{v_{1}}{\delta}\right)  \right\Vert
_{W^{s,p}\left(  \mathbf{R}\right)  }\left\Vert \left\Vert g\right\Vert
_{_{W_{v_{1}}^{\frac{1}{p},p}\left(  \mathbf{R}\right)  }}+\left\Vert
g\right\Vert _{_{L_{v_{1}}^{\infty}\left(  \mathbf{R}\right)  }}\right\Vert
_{L_{v_{2}}^{p}}\\
&  \leq C\left\Vert f\left(  \frac{v_{1}}{\delta}\right)  \right\Vert
_{W^{s,p}\left(  \mathbf{R}\right)  }\left\Vert \left\Vert g\right\Vert
_{_{W_{v_{1}}^{1,p}\left(  \mathbf{R}\right)  }}\right\Vert _{L_{v_{2}}^{p}}\\
&  \leq C\left\Vert f\left(  \frac{v_{1}}{\delta}\right)  \right\Vert
_{W^{s,p}\left(  \mathbf{R}\right)  }\left\Vert g\right\Vert _{W^{1,p}\left(
\mathbf{R}^{2}\right)  }\text{ }\rightarrow0,\
\end{align*}
when $\delta\rightarrow0.$Since $\left\Vert f\left(  \frac{v_{1}}{\delta
}\right)  \right\Vert _{W^{s,p}\left(  \mathbf{R}\right)  }\rightarrow0$ under
the assumption $s<\frac{1}{p}$ (see \cite{lin-zeng-bgk} for a proof). By the
trace Theorem, we also have
\[
\left\Vert f\left(  \frac{v_{1}}{\delta}\right)  \left\Vert g\left(
v_{1},v_{2}\right)  \right\Vert _{W_{v_{2}}^{s,p}\left(  \mathbf{R}\right)
}\right\Vert _{L_{v_{1}}^{p}}\leq C\left\Vert f\left(  \frac{v_{1}}{\delta
}\right)  \right\Vert _{L^{p}}\left\Vert g\right\Vert _{W^{2,p}\left(
\mathbf{R}^{2}\right)  }\rightarrow0,\
\]
when $\delta\rightarrow0$. This proves (\ref{zero-lemma}) for $g\in
C_{0}^{\infty}\left(  \mathbf{R}^{2}\right)  $. When $g\in W^{s,p}\left(
\mathbf{R}^{2}\right)  $, (\ref{zero-lemma}) can be proved by using
$C_{0}^{\infty}\left(  \mathbf{R}^{2}\right)  $ functions as approximations.

Proof of (ii): By Fubini Theorem for $W^{s,p}\left(  \mathbf{R}^{2}\right)  $
norm,
\begin{align*}
&  \left\Vert f\left(  \frac{v_{1}}{\delta}\right)  g\left(  v_{2}\right)
\right\Vert _{W^{s,p}\left(  \mathbf{R}^{2}\right)  }\\
&  \leq C\left(  \left\Vert f\left(  \frac{v_{1}}{\delta}\right)  \right\Vert
_{W_{v_{1}}^{s,p}\left(  \mathbf{R}\right)  }\left\Vert g\left(  v_{2}\right)
\right\Vert _{L_{v_{2}}^{p}}+\left\Vert g\left(  v_{2}\right)  \right\Vert
_{W_{v_{2}}^{s,p}\left(  \mathbf{R}\right)  }\left\Vert f\left(  \frac{v_{1}%
}{\delta}\right)  \right\Vert _{L_{v_{1}}^{p}}\right) \\
&  \rightarrow0\text{, when }\delta\rightarrow0.
\end{align*}

\end{proof}

By the similar proof of Theorem \ref{thm-existence}, we can get the following.

\begin{theorem}
\label{Thm-existence-weighted}Assume the homogeneous distribution function%
\[
f_{0}\left(  v\right)  \in H^{s_{v},b}\left(  \mathbf{R}^{d}\right)
\ \ \left(  d\geq2,\ b>\frac{d-1}{4},\ s_{v}\in\lbrack0,\frac{3}{2})\ \right)
\]
satisfies%
\[
f_{0}\left(  v\right)  \geq0,\ \int f_{0}\left(  v\right)  dv=1,\ \int
v^{2}f_{0}\left(  v\right)  dv<+\infty.
\]
Fix $T_{1}>0$ and $c\in\mathbf{R}$. Then for any $\varepsilon>0$, $s_{x}%
\geq0,\ $there exist travelling wave solutions of the form $f=f_{\varepsilon
}\left(  x_{1}-ct,v\right)  ,$ $\vec{E}=E_{\varepsilon}\left(  x_{1}%
-ct\right)  \vec{e}_{1}$ to (\ref{vpe}), such that$\ \left(  f_{\varepsilon
}\left(  x_{1},v\right)  ,E_{\varepsilon}\left(  x_{1}\right)  \right)  $ has
minimal period $T_{1}$ in $x_{1}$,$\ f_{\varepsilon}\left(  x_{1},v\right)
\geq0,$ $E_{\varepsilon}\left(  x_{1}\right)  $ is not identically zero, and
\begin{equation}
\ \left\Vert f_{\varepsilon}-f_{0}\right\Vert _{L_{x_{1},v}^{1}}+\ \int
_{0}^{T_{1}}\int_{\mathbf{R}^{d}}\left\vert v\right\vert ^{2}\left\vert
f_{\varepsilon}\left(  x_{1},v\right)  -\ f_{0}\left(  v\right)  \right\vert
dx_{1}dv+\left\Vert f_{\varepsilon}-f_{0}\right\Vert _{H_{x}^{s_{x}}%
H_{v}^{s_{v},b}}<\varepsilon.\ \label{BGK-norm-weighted}%
\end{equation}

\end{theorem}

\begin{proof}
The construction of BGK waves follows the same line of the proof of Theorem
\ref{thm-existence}. First, we modify $f_{0}\left(  v\right)  $ to a smooth
profile $f_{1}\left(  v\right)  $. Then by adding proper perturbations in a
scaling form to $f_{1}\left(  v\right)  $, we get the modified profile
$f_{\gamma,\delta}\left(  v\right)  .$ The BGK waves $\left(  f_{\varepsilon
}\left(  x_{1},v\right)  ,E_{\varepsilon}\left(  x_{1}\right)  \right)  \ $are
obtained by bifurcation near $\left(  f_{\gamma,\delta}\left(  v\right)
,0\right)  .$ To show the estimate (\ref{BGK-norm-weighted}), we need to
control three deviations in the norm of (\ref{BGK-norm-weighted}): i)
$f_{\varepsilon}\left(  x_{1},v\right)  -f_{\gamma,\delta}\left(  v\right)
;\ $ii) $f_{\gamma,\delta}\left(  v\right)  -f_{1}\left(  v\right)  ;$ and
iii) $f_{1}\left(  v\right)  -f_{0}\left(  v\right)  $. For the estimate of
i), we choose integers $\bar{s}_{x}\geq s_{x},\ \bar{s}_{v}\geq s_{v},$ and
$\bar{b}\geq b$ and it is easy to show that
\[
\left\Vert f_{\varepsilon}\left(  x_{1},v\right)  -f_{\gamma,\delta}\left(
v\right)  \right\Vert _{H_{x}^{s_{x}}H_{v}^{s_{v},b}}\leq C\left\Vert
f_{\varepsilon}\left(  x_{1},v\right)  -f_{\gamma,\delta}\left(  v\right)
\right\Vert _{H_{x}^{\bar{s}_{x}}H_{v}^{\bar{s}_{v},\bar{b}}}%
\]
and the right hand side can be made arbitrarily small by using the dominant
convergence Theorem. For estimates of ii) and iii), we use the following
analogue of Lemma \ref{lemma-delta-0}.
\end{proof}

\begin{lemma}
\label{lemma-delta-0-H}

(i) Given $f\left(  v_{1}\right)  \in H^{\frac{1}{2}}\left(  \mathbf{R}%
\right)  \cap L^{\infty}\left(  \mathbf{R}\right)  ,\ g\left(  v_{1}%
,v_{2}\right)  \in H^{s,b}\left(  \mathbf{R}^{2}\right)  \ \left(  0\leq
s<\frac{1}{2},b>\frac{1}{4}\right)  $. For $\delta>0,\ $
\[
\left\Vert f\left(  \frac{v_{1}}{\delta}\right)  g\left(  v_{1},v_{2}\right)
\right\Vert _{H^{s,b}\left(  \mathbf{R}^{2}\right)  }\rightarrow0\text{, when
}\delta\rightarrow0\text{. }%
\]

(ii) Given $f,g\in H^{s,b}\left(  \mathbf{R}\right)  $ $\left(  0\leq
s<\frac{1}{2},b>\frac{1}{4}\right)  $.~For $\delta>0,\ $
\[
\left\Vert f\left(  \frac{v_{1}}{\delta}\right)  g\left(  v_{2}\right)
\right\Vert _{H^{s,b}\left(  \mathbf{R}^{2}\right)  }\rightarrow0\text{, when
}\delta\rightarrow0\text{. }%
\]

\end{lemma}

\begin{proof}
First, we show that for any function $h\in H^{s,b}\left(  \mathbf{R}%
^{d}\right)  \ \left(  d\geq1,0\leq s\leq2,b>\frac{1}{4}\right)  ,$ the norm
$\left\Vert h\right\Vert _{H^{s,b}\left(  \mathbf{R}^{d}\right)  }$ defined by
(\ref{definition-H-tilde}) is equivalent to both%
\begin{equation}
\left\Vert \left(  1+\left\vert v\right\vert ^{2}\right)  ^{b}f\right\Vert
_{H^{s}\left(  \mathbf{R}^{d}\right)  } \label{norm-equi-1}%
\end{equation}
$\ $and
\begin{equation}
\left\Vert \left(  1+\left\vert v_{1}\right\vert ^{2b}+\cdots+\left\vert
v_{d}\right\vert ^{2b}\right)  f\right\Vert _{H^{s}\left(  \mathbf{R}%
^{d}\right)  }. \label{norm-equi-2}%
\end{equation}
We only need to prove the equivalence of the norms (\ref{definition-H-tilde})
and (\ref{norm-equi-1}) for $s=0$ and $s=2$, since then for $0<s<2$ it follows
from interpolation. For $s=0,$ it is trivial. For $s=2,$ by choosing $a>0$
small enough, we have
\begin{align*}
&  \left\Vert \left(  1+a\left\vert v\right\vert ^{2}\right)  ^{b}\left(
1-\Delta\right)  f-\left(  1-\Delta\right)  \left(  1+a\left\vert v\right\vert
^{2}\right)  ^{b}f\right\Vert _{L^{2}\left(  \mathbf{R}^{d}\right)  }\\
&  =\left\Vert f\ \Delta\left(  \left(  1+a\left\vert v\right\vert
^{2}\right)  ^{b}-1\right)  +2\nabla f\cdot\nabla\left(  \left(  1+a\left\vert
v\right\vert ^{2}\right)  ^{b}-1\right)  \right\Vert _{L^{2}}\\
&  \leq\frac{1}{2}\left\Vert \left(  1+a\left\vert v\right\vert ^{2}\right)
^{b}\left(  1-\Delta\right)  f\right\Vert _{L^{2}}.
\end{align*}
Thus
\begin{align*}
\frac{1}{2}\left\Vert \left(  1+a\left\vert v\right\vert ^{2}\right)
^{b}\left(  1-\Delta\right)  f\right\Vert _{L^{2}}  &  \leq\left\Vert \left(
1-\Delta\right)  \left(  1+a\left\vert v\right\vert ^{2}\right)
^{b}f\right\Vert _{L^{2}}\\
&  \leq\frac{3}{2}\left\Vert \left(  1+a\left\vert v\right\vert ^{2}\right)
^{b}\left(  1-\Delta\right)  f\right\Vert _{L^{2}}.
\end{align*}
The equivalence of (\ref{definition-H-tilde}) and (\ref{norm-equi-2}) can be
proved in the same way.

Proof of (i): By Lemma \ref{lemma-delta-0} (i),
\[
\left\Vert f\left(  \frac{v_{1}}{\delta}\right)  g\left(  v_{1},v_{2}\right)
\right\Vert _{H^{s,b}\left(  \mathbf{R}^{2}\right)  }\leq C\left\Vert \left(
1+\left\vert v\right\vert ^{2}\right)  ^{b}f\left(  \frac{v_{1}}{\delta
}\right)  g\left(  v_{1},v_{2}\right)  \right\Vert _{H^{s}\left(
\mathbf{R}^{2}\right)  }\rightarrow0,
\]
when $\delta\rightarrow0$. Since $f\left(  v_{1}\right)  \in H^{\frac{1}{2}%
,}\left(  \mathbf{R}\right)  \cap L^{\infty}\left(  \mathbf{R}\right)  $ and
\[
\left\Vert \left(  1+\left\vert v\right\vert ^{2}\right)  ^{b}g\left(
v_{1},v_{2}\right)  \right\Vert _{H^{s}}\leq C\left\Vert g\right\Vert
_{H^{s,b}\left(  \mathbf{R}^{2}\right)  }<\infty.
\]

Proof of (ii): By using the equivalent norm (\ref{norm-equi-2}) and Lemma
\ref{lemma-delta-0} (ii),%
\begin{align*}
&  \left\Vert f\left(  \frac{v_{1}}{\delta}\right)  g\left(  v_{2}\right)
\right\Vert _{H^{s,b}\left(  \mathbf{R}^{2}\right)  }\\
&  \leq C\left\Vert \left(  1+\left\vert v_{1}\right\vert ^{2b}+\left\vert
v_{2}\right\vert ^{2b}\right)  f\left(  \frac{v_{1}}{\delta}\right)  g\left(
v_{2}\right)  \right\Vert _{H^{s}\left(  \mathbf{R}^{2}\right)  }\\
&  \leq C\left(  \left\Vert f\left(  \frac{v_{1}}{\delta}\right)  g\left(
v_{2}\right)  \right\Vert _{H^{s}\left(  \mathbf{R}^{2}\right)  }+\delta
^{2b}\left\Vert \left\vert \frac{v_{1}}{\delta}\right\vert ^{2b}f\left(
\frac{v_{1}}{\delta}\right)  g\left(  v_{2}\right)  \right\Vert _{H^{s}\left(
\mathbf{R}^{2}\right)  }+\left\Vert f\left(  \frac{v_{1}}{\delta}\right)
\left\vert v_{2}\right\vert ^{2b}g\left(  v_{2}\right)  \right\Vert
_{H^{s}\left(  \mathbf{R}^{2}\right)  }\right) \\
&  \rightarrow0,\text{ when }\delta\rightarrow0.
\end{align*}

\end{proof}

In the following, we show that there exist no truly $2$D or $3$D BGK
solutions. Therefore, the $1D$ BGK form of solutions constructed in Theorem
\ref{thm-existence} is in some sense necessary.

\begin{proposition}
\label{prop-non-bgk}(i) $\left(  d=2\right)  $ Assume $\mu\in C^{1}\left(
\mathbf{R}\right)  \cap L^{1}\left(  \mathbf{R}^{+}\right)  ,\ \mu\geq0.$ If
\[
f_{0}\left(  x,v\right)  =\mu\left(  \frac{1}{2}\left\vert v\right\vert
^{2}-\beta\left(  x\right)  \right)  ,\ \ \ \ \vec{E}_{0}\left(  x\right)
=-\nabla\beta
\]
is a solution of the Vlasov-Poisson system, then $\vec{E}_{0}\equiv0.$

(ii) $\left(  d=3\right)  $ Assume $\mu\in C^{1}\left(  \mathbf{R}\right)
\cap L^{1}\left(  \mathbf{R}^{+}\right)  ,\ \mu\geq0.\ $If
\[
f_{0}\left(  x_{1},x_{2},v_{1},v_{2},v_{3}\right)  =\mu\left(  \frac{1}%
{2}\left(  v_{1}^{2}+v_{2}^{2}\right)  -\beta\left(  x_{1},x_{2}\right)
,v_{3}\right)  ,\ \ \ \ \vec{E}_{0}\left(  x_{1},x_{2}\right)  =\left(
-\partial_{x_{1}}\beta,-\partial_{x_{2}}\beta,0\right)
\]
is a solution of the Vlasov-Poisson system, then $\vec{E}_{0}\equiv0.$

(iii) $\left(  d=3\right)  \ $Assume $\mu\in C^{1}\left(  \mathbf{R}\right)
,\ \mu\geq0,\ \mu\left(  r\right)  \sqrt{r}\in L^{1}\left(  \mathbf{R}%
^{+}\right)  .\ $If
\[
f_{0}\left(  x,v\right)  =\mu\left(  \frac{1}{2}\left\vert v\right\vert
^{2}-\beta\left(  x\right)  \right)  ,\ \ \ \ \vec{E}_{0}\left(  x\right)
=-\nabla\beta
\]
is a solution of the Vlasov-Poisson system, then $\vec{E}_{0}\equiv0.$
\end{proposition}

\begin{proof}
We only prove (i) since the proof of (ii) and (iii) is similar. The electric
potential $\beta$ satisfies
\begin{equation}
-\Delta\beta=-\int_{\mathbf{R}^{2}}\mu\left(  \frac{1}{2}\left\vert
v\right\vert ^{2}-\beta\right)  \ dv+1=g\left(  \beta\right)  .
\label{eqn-beta}%
\end{equation}
By the assumptions on $\mu$, we have $g\left(  \beta\right)  \in C^{1}\left(
\mathbf{R}\right)  $ and
\begin{align*}
g^{\prime}\left(  \beta\right)   &  =-\int_{\mathbf{R}^{2}}\mu^{\prime}\left(
\frac{1}{2}\left\vert v\right\vert ^{2}-\beta\left(  x\right)  \right)  dv\ \\
&  =2\pi\int_{0}^{\infty}\mu^{\prime}\left(  s-\beta\right)  ds=-2\pi
\mu\left(  -\beta\right)  \leq0,
\end{align*}
Taking $x_{1}$ derivative of (\ref{eqn-beta}) and integrating with
$\beta_{x_{1}}$, we have
\[
\int_{T^{2}}\left\vert \nabla\beta_{x_{1}}\right\vert ^{2}\ dx=\int
_{\mathbf{T}^{2}}g^{\prime}\left(  \beta\right)  \left\vert \beta_{x_{1}%
}\right\vert ^{2}dx\leq0.
\]
So $\int_{T^{2}}\left\vert \nabla\beta_{x_{1}}\right\vert ^{2}\ dx=0$ and
$\beta_{x_{1}}$ is a constant $C$. By the periodic assumption of $\beta$,
$C=0$ and thus $\beta_{x_{1}}\equiv0.$ Similarly, $\beta_{x_{2}}\equiv0.$
\end{proof}

\begin{remark}
In $2$D and $3$D, the function $g\left(  \beta\right)  $ defined in
(\ref{eqn-beta}) always satisfies $g^{\prime}\left(  \beta\right)  \leq0$ and
thus the elliptic problem (\ref{eqn-beta}) only has trivial solutions. For
$1$D, the function $g^{\prime}\left(  \beta\right)  $ can change signs and
thus the existence of $1$D BGK waves is possible. We note that Proposition
\ref{prop-non-bgk} does not exclude steady (travelling wave) solutions in $2$D
and $3$D, which are not of BGK types. It would be interesting to construct or
exclude nontrivial steady solutions not of BGK types.
\end{remark}

\section{Invariant structures in $H_{x}^{s_{x}}H_{v}^{s_{v},b}$ $\left(
s>\frac{3}{2}\right)  $}

First, we prove a technical lemma to be used later.

\begin{lemma}
\label{lemma-proj-inequality}Given $f\left(  v\right)  \in H^{s,b}\left(
\mathbf{R}^{d}\right)  \ \ \left(  d\geq2,\ s\geq0,\ b>\frac{d-1}{4}\right)
.$ For any unit vector $\vec{e}\in\mathbf{R}^{d}$, let $v=\alpha\vec{e}+w$
where $v\in\mathbf{R}^{d}$ and $w\perp\vec{e}$. Define
\begin{equation}
f_{\vec{e}}\left(  \alpha\right)  =\int_{\mathbf{R}^{d-1}}f\left(  \alpha
\vec{e}+w\right)  \ dw. \label{defn-f-e-arrow}%
\end{equation}
Then
\[
\left\Vert f_{\vec{e}}\left(  \alpha\right)  \right\Vert _{H^{s}\left(
\mathbf{R}\right)  }\leq C\left\Vert f\right\Vert _{H^{s,b}\left(
\mathbf{R}^{d}\right)  }%
\]
for some constant $C$ independent of $\vec{e}.$
\end{lemma}

\begin{proof}
To simplify notations, we only consider $\vec{e}=\left(  1,0,\cdots,0\right)
$. Then $\alpha=v_{1}$ and
\[
f_{\vec{e}}\left(  v_{1}\right)  =\int_{\mathbf{R}^{d-1}}f\left(  v\right)
\ dv_{2}\cdots dv_{d}.
\]
Let $\xi=\left(  \xi_{1},\cdots\xi_{d}\right)  $ be the Fourier variable.
Then
\begin{align*}
\left\Vert f_{\vec{e}}\left(  v_{1}\right)  \right\Vert _{H^{s}\left(
\mathbf{R}\right)  }  &  =\left\Vert \left(  1+\left\vert \xi_{1}\right\vert
^{2}\right)  ^{\frac{s}{2}}\hat{f}_{\vec{e}}\left(  \xi_{1}\right)
\right\Vert _{L^{2}\left(  \mathbf{R}\right)  }\\
&  =\left\Vert \left(  1+\left\vert \xi_{1}\right\vert ^{2}\right)  ^{\frac
{s}{2}}\hat{f}\left(  \xi_{1},0,\cdots,0\right)  \right\Vert _{L^{2}\left(
\mathbf{R}\right)  }\\
&  \leq C\left\Vert \left(  1+\left\vert \xi\right\vert ^{2}\right)
^{\frac{s}{2}}\hat{f}\left(  \xi\right)  \right\Vert _{H^{2b}\left(
\mathbf{R}^{d}\right)  }\text{ }\\
&  =C\left\Vert \left(  1+\left\vert v\right\vert ^{2}\right)  ^{b}\left(
1-\Delta\right)  ^{\frac{s}{2}}f\right\Vert _{L^{2}\left(  \mathbf{R}%
^{d}\right)  }=C\left\Vert f\right\Vert _{H^{s,b}\left(  \mathbf{R}%
^{d}\right)  }.
\end{align*}
Here, the first inequality above is due to the trace theorem and that
$2b>\frac{d-1}{2}.$
\end{proof}

\begin{corollary}
\label{cor-inequality}Given $f\left(  v\right)  \in H^{s,b}\left(
\mathbf{R}^{d}\right)  \ \ \left(  d\geq2,\ s>\frac{3}{2},\ b>\frac{d-1}%
{4}\right)  .$ For any unit vector $\vec{e}\in\mathbf{R}^{d}$, we have

(i) If $\alpha_{0}$ is a critical point of $f_{\vec{e}}\left(  \alpha\right)
$, then
\[
\int_{\mathbf{R}}\left\vert \frac{f_{\vec{e}}^{\prime}\left(  \alpha\right)
}{\alpha-\alpha_{0}}\right\vert d\alpha\leq C\left(  d,s,b\right)  \left\Vert
f\right\Vert _{H^{s,b}\left(  \mathbf{R}^{d}\right)  }.
\]

(ii) For any $\alpha^{\prime}\in\mathbf{R}$,
\[
\left\vert P\int_{\mathbf{R}}\frac{f_{\vec{e}}^{\prime}\left(  \alpha\right)
}{\alpha-\alpha^{\prime}}d\alpha\right\vert \leq C\left(  d,s,b\right)
\left\Vert f\right\Vert _{H^{s,b}\left(  \mathbf{R}^{d}\right)  },
\]
where $P\int_{\mathbf{R}}$ is the principal value integral.
\end{corollary}

\begin{proof}
(i) follows from Lemma \ref{lemma-proj-inequality} and the following Hardy
inequality (see Lemma 3.1 of \cite{lin-zeng-bgk}): If $u\left(  v\right)  \in
W^{s,p}\left(  \mathbf{R}\right)  $ $\left(  p>1,s>\frac{1}{p}\right)  ,\ $and
$u\left(  0\right)  =0,\ $then%
\begin{equation}
\int_{\mathbf{R}}\left\vert \frac{u\left(  v\right)  }{v}\right\vert dv\leq
C\left\Vert u\right\Vert _{W^{s,p}\left(  \mathbf{R}\right)  },
\label{hardy-inequality}%
\end{equation}
for some constant $C$.

Proof of (ii): Since
\[
P\int_{\mathbf{R}}\frac{f_{\vec{e}}^{\prime}\left(  \alpha\right)  }%
{\alpha-\alpha^{\prime}}d\alpha=\frac{1}{2}\int_{\mathbf{R}}\frac{\frac
{d}{d\alpha}\left(  f_{\vec{e}}\left(  \alpha+\alpha^{\prime}\right)
+f_{\vec{e}}\left(  -\alpha+\alpha^{\prime}\right)  \right)  }{\alpha}%
d\alpha,
\]
by Hardy inequality (\ref{hardy-inequality})
\begin{align*}
\left\vert P\int_{\mathbf{R}}\frac{f_{\vec{e}}^{\prime}\left(  \alpha\right)
}{\alpha-\alpha^{\prime}}d\alpha\right\vert  &  \leq C\left(  \left\Vert
f_{\vec{e}}\left(  \alpha+\alpha^{\prime}\right)  \right\Vert _{H^{s}\left(
\mathbf{R}\right)  }+\left\Vert f_{\vec{e}}\left(  -\alpha+\alpha^{\prime
}\right)  \right\Vert _{H^{s}\left(  \mathbf{R}\right)  }\right) \\
&  \leq C\left\Vert f_{\vec{e}}\left(  \alpha\right)  \right\Vert
_{H^{s}\left(  \mathbf{R}\right)  }\leq C\left\Vert f\right\Vert
_{H^{s,b}\left(  \mathbf{R}^{d}\right)  }\text{ (by Lemma
\ref{lemma-proj-inequality}).}%
\end{align*}

\end{proof}

Next we derive the linear decay estimate in higher dimensions. We start with a
generalization of Penrose's linear stability condition: Given $f_{0}\left(
v\right)  \in H^{s,b}\left(  \mathbf{R}^{d}\right)  $ $\left(  d\geq
2,\ s>\frac{3}{2},\ b>\frac{d-1}{4}\right)  ,$%

\begin{equation}
\left\vert \vec{k}\right\vert ^{2}-\max_{v_{i}\in S_{\vec{k}/\left\vert
\vec{k}\right\vert }}\int_{\mathbf{R}}\frac{f_{0,\vec{k}/\left\vert \vec
{k}\right\vert }^{\prime}\left(  \alpha\right)  }{\alpha-v_{i}}d\alpha
>0,\ \text{for any }\vec{k}\in Z^{d}, \label{penrose-condition}%
\end{equation}
where $f_{0,\vec{k}/\left\vert \vec{k}\right\vert }^{\prime}\left(
\alpha\right)  $ is defined by (\ref{defn-f-e-arrow}) and $S_{\vec
{k}/\left\vert \vec{k}\right\vert }$ is the set of all critical points of
$f_{0,\vec{k}/\left\vert \vec{k}\right\vert }\left(  \alpha\right)  .$

\begin{remark}
By Corollary \ref{cor-inequality}, one only need to check the stability
condition (\ref{penrose-condition}) for finitely many $\vec{k}\in Z^{d}$
satisfying that
\[
\left\vert \vec{k}\right\vert ^{2}\leq C\left(  d,s,b\right)  \left\Vert
f_{0}\right\Vert _{H^{s,b}\left(  \mathbf{R}^{d}\right)  }.
\]
In particular, for a single humped isentropic profile $f_{0}\left(  v\right)
=\mu\left(  \frac{1}{2}\left\vert v\right\vert ^{2}\right)  $ with
$\mu^{\prime}\left(  e\right)  <0$, the stability condition
(\ref{penrose-condition}) is satisfied for any period set $\left(
T_{1},\cdots,T_{d}\right)  .$
\end{remark}

The next lemma is the linear decay estimate in a space-time norm, which
generalizes the one dimensional result in \cite{lin-zeng-bgk}. The linearized
Vlasov-Poisson system at an homogeneous state $\left(  f,\vec{E}\right)
=\left(  f_{0}\left(  v\right)  ,\vec{0}\right)  $ is
\begin{subequations}
\begin{equation}
\partial_{t}f+v\cdot\nabla_{x}f-\vec{E}\cdot\nabla_{v}f_{0}=0,
\label{L-Vlasov}%
\end{equation}%
\begin{equation}
\,\,\,\,\vec{E}=-\nabla_{x}\phi,\ \ -\Delta\phi=-\int_{\mathbf{R}^{d}}f\ dv,
\label{L-Poisson}%
\end{equation}

\end{subequations}
\begin{lemma}
\label{lemma-estimate-linear}Assume $f_{0}\left(  v\right)  \in H^{s_{0}%
,b}\left(  \mathbf{R}^{d}\right)  $ $\left(  d\geq2,\ s_{0}>\frac{3}%
{2},\ b>\frac{d-1}{4}\right)  \ $and the Penrose stability condition
(\ref{penrose-condition}) is satisfied for $x-$period tuple $\left(
T_{1},\cdots,T_{d}\right)  $. Let $\left(  f\left(  x,v,t\right)  ,\vec
{E}\left(  x,t\right)  \right)  $ be a solution of the linearized
Vlasov-Poisson system (\ref{L-Vlasov})-(\ref{L-Poisson}) with $x-$period tuple
$\left(  T_{1},\cdots,T_{d}\right)  $. If$\ g\in H_{x}^{s_{x}}H_{v}^{s_{v},b}$
with $\left\vert s_{v}\right\vert \leq s_{0}-1,$ then%
\begin{equation}
\left\Vert t^{s_{v}}\vec{E}\left(  x,t\right)  \right\Vert _{L_{t}^{2}%
H_{x}^{\frac{3}{2}+s_{x}+s_{v}}}\leq C_{0}\left\Vert f\left(  x,v,0\right)
\right\Vert _{H_{x}^{s_{x}}H_{v}^{s_{v},b}}, \label{estimate-linear-integral}%
\end{equation}

\end{lemma}

\begin{proof}
First, we reduce the linearized problem to the one dimensional case. Since the
homogeneous component of $f\left(  x,v,t\right)  \ $remains steady for the
linearized equation and therefore has no effect on $\vec{E}\left(  x,t\right)
$, we assume that $f$ has no homogeneous component. Let
\[
f\left(  x,v,t\right)  =\sum_{\vec{0}\neq\vec{k}\in\mathbf{Z}^{d}}e^{i\vec
{k}\cdot x}f_{\vec{k}}\left(  v,t\right)
\]
and the electric potential
\[
\phi\left(  x,t\right)  =\sum_{\vec{0}\neq\vec{k}\in\mathbf{Z}^{d}}e^{i\vec
{k}\cdot x}\phi_{\vec{k}}\left(  t\right)  .
\]
Then
\[
\vec{E}\left(  x,t\right)  =-\nabla_{x}\phi=-\sum_{\vec{0}\neq\vec{k}%
\in\mathbf{Z}^{d}}i\vec{k}\phi_{\vec{k}}\left(  t\right)  e^{i\vec{k}\cdot
x}=\sum_{\vec{0}\neq\vec{k}\in\mathbf{Z}^{d}}\vec{E}_{\vec{k}}\left(
t\right)  e^{i\vec{k}\cdot x},
\]
where $\vec{E}_{\vec{k}}\left(  t\right)  =-i\vec{k}\phi_{\vec{k}}\left(
t\right)  $. Denote $\vec{e}=\vec{k}/\left\vert \vec{k}\right\vert $, then
\[
\vec{E}_{\vec{k}}\left(  t\right)  =-i\vec{k}\phi_{\vec{k}}\left(  t\right)
=\tilde{E}_{\vec{k}}\left(  t\right)  \vec{e}%
\]
where $\tilde{E}_{\vec{k}}\left(  t\right)  =-i\left\vert \vec{k}\right\vert
\phi_{\vec{k}}\left(  t\right)  $. Let $v=\alpha\vec{e}+w$ where $\alpha
\in\mathbf{R,\ }w\perp\vec{e}\,$, and
\[
\tilde{f}_{\vec{k}}\left(  \alpha,t\right)  =f_{\vec{k},\vec{e}}\left(
\alpha,t\right)  =\int_{\mathbf{R}^{d-1}}f_{\vec{k}}\left(  \alpha\vec
{e}+w,t\right)  \ dw
\]
The linearized Vlasov equation implies that
\begin{align*}
0  &  =\partial_{t}f_{\vec{k}}+v\cdot i\vec{k}\ f_{\vec{k}}-\vec{E}_{\vec{k}%
}\cdot\nabla_{v}f_{0}\\
&  =\partial_{t}f_{\vec{k}}+i\alpha\left\vert \vec{k}\right\vert f_{\vec{k}%
}-\tilde{E}_{k}\partial_{\alpha}f_{0}%
\end{align*}
An integration of the $w$ variable on above equation yields%
\begin{equation}
\partial_{t}\tilde{f}_{\vec{k}}\left(  \alpha,t\right)  +i\alpha\left\vert
\vec{k}\right\vert \tilde{f}_{\vec{k}}\left(  \alpha,t\right)  -\tilde{E}%
_{k}f_{0,\vec{e}}^{\prime}\left(  \alpha\right)  =0. \label{L-Vlasov-k}%
\end{equation}
The Poisson equation implies%
\[
\left\vert \vec{k}\right\vert ^{2}\phi_{\vec{k}}\left(  t\right)
=-\int_{\mathbf{R}^{d}}f_{\vec{k}}\left(  v,t\right)  dv,
\]
and thus
\begin{equation}
i\left\vert \vec{k}\right\vert \tilde{E}_{\vec{k}}\left(  t\right)
=-\int_{\mathbf{R}}\tilde{f}_{\vec{k}}\left(  \alpha,t\right)  d\alpha.
\label{L-Poisson-k}%
\end{equation}
Equations (\ref{L-Vlasov-k}) and (\ref{L-Poisson-k}) imply that $\left(
\tilde{f}_{\vec{k}}\left(  \alpha,t\right)  ,\tilde{E}_{k}\left(  t\right)
\right)  e^{i\left\vert \vec{k}\right\vert x}$ solves the linearized $1$D
Vlasov-Poisson equations at the homogeneous profile $f_{0,\vec{e}}\left(
\alpha\right)  $. Thus by the $1$D representation formula in
\cite{lin-zeng-bgk} and the Penrose stability condition
(\ref{penrose-condition}), we have
\[
\tilde{E}_{\vec{k}}\left(  t\right)  =\frac{\left\vert \vec{k}\right\vert
}{2\pi}\int_{\mathbf{R}}\frac{G_{\vec{k}}\left(  y+i0\right)  }{\left\vert
\vec{k}\right\vert ^{2}-F_{\vec{e}}\left(  y+i0\right)  }e^{-i\left\vert
\vec{k}\right\vert yt}dy.
\]
Here,
\[
G_{\vec{k}}\left(  y+i0\right)  =P\int_{\mathbf{R}}\frac{\tilde{f}_{\vec{k}%
}\left(  \alpha,0\right)  }{\alpha-y}d\alpha+i\pi\tilde{f}_{\vec{k}}\left(
y,0\right)
\]
and
\[
F_{\vec{e}}\left(  y+i0\right)  =P\int_{\mathbf{R}}\frac{f_{0,\vec{e}}%
^{\prime}\left(  \alpha\right)  }{\alpha-y}d\alpha+i\pi f_{0,\vec{e}}^{\prime
}\left(  y\right)  .
\]
By the Penrose stability condition (\ref{penrose-condition}) and
\[
\left\vert P\int_{\mathbf{R}}\frac{f_{0,\vec{e}}^{\prime}\left(
\alpha\right)  }{\alpha-y}d\alpha\right\vert \leq C\left(  d,s,b\right)
\left\Vert f_{0}\right\Vert _{H^{s,b}\left(  \mathbf{R}^{d}\right)  }\text{
(Corollary \ref{cor-inequality})},
\]
there exists $c_{0}>0$ (independent of $\vec{k}$), such that
\[
\left\vert \left\vert \vec{k}\right\vert ^{2}-F_{\vec{e}}\left(  y+i0\right)
\right\vert ^{2}\geq c_{0}\left\vert \vec{k}\right\vert ^{2}.
\]
Then by the same proof of Proposition 4.1 in \cite{lin-zeng-bgk},
\begin{align*}
\left\Vert t^{s_{v}}\tilde{E}_{\vec{k}}\left(  t\right)  \right\Vert _{L^{2}%
}^{2}  &  \leq C\left\vert \vec{k}\right\vert ^{-3-2s_{v}}\left\Vert \tilde
{f}_{\vec{k}}\left(  \alpha,0\right)  \right\Vert _{H_{v}^{s_{v}}}^{2}\\
&  \leq C\left\vert \vec{k}\right\vert ^{-3-2s_{v}}\left\Vert f_{\vec{k}%
}\left(  v,0\right)  \right\Vert _{H^{s,b}\left(  \mathbf{R}^{d}\right)  }%
^{2}.
\end{align*}
So
\begin{align*}
&  \ \ \ \ \ \left\Vert t^{s_{v}}\vec{E}\left(  x,t\right)  \right\Vert
_{L_{t}^{2}H_{x}^{\frac{3}{2}+s_{x}+s_{v}}}^{2}\\
&  =\sum_{0\neq\vec{k}\in\mathbf{Z}^{d}}\left\vert \vec{k}\right\vert
^{3+2s_{v}+2s_{x}}\left\Vert t^{s_{v}}\tilde{E}_{\vec{k}}\left(  t\right)
\right\Vert _{L^{2}}^{2}\\
&  \leq C\sum_{0\neq\vec{k}\in\mathbf{Z}^{d}}\left\vert \vec{k}\right\vert
^{2s_{x}}\left\Vert f_{\vec{k}}\left(  v,0\right)  \right\Vert _{H^{s,b}%
\left(  \mathbf{R}^{d}\right)  }^{2}\\
&  \leq C\left\Vert f\left(  x,v,0\right)  \right\Vert _{H_{x}^{s_{x}}%
H_{v}^{s_{v},b}}.
\end{align*}

\end{proof}

\begin{lemma}
\label{lemma-estimate-integral-small}Assume $f_{0}\left(  v\right)  \in
H^{s_{0},b}\left(  \mathbf{R}^{d}\right)  $ $\left(  d\geq2,\ s_{0}>\frac
{3}{2},\ b>\frac{d-1}{4}\right)  \ $and the Penrose stability condition
(\ref{penrose-condition}) is satisfied for $x-$period tuple $\left(
T_{1},\cdots,T_{d}\right)  $. Let $\left(  f\left(  x,v,t\right)  ,\vec
{E}\left(  x,t\right)  \right)  $ be a solution of the Vlasov-Poisson system
(\ref{vlasov})-(\ref{poisson}) with $x-$period tuple $\left(  T_{1}%
,\cdots,T_{d}\right)  $.

For any $\left(  s_{x},s_{v}\right)  $ satisfying
(\ref{condition-non-existence}),$\ $there exists $\varepsilon_{0}>0$, such
that if
\[
\left\Vert f\left(  t\right)  -f_{0}\right\Vert _{H_{x}^{s_{x}}H_{v}^{s_{v}%
,b}}<\varepsilon_{0},\ \text{for\ all\ }t\geq0,
\]
then%
\begin{equation}
\left\Vert \left(  1+t\right)  ^{s_{v}-1}\vec{E}\left(  x,t\right)
\right\Vert _{L_{\left\{  t\geq0\right\}  }^{2}H_{x}^{\frac{3}{2}+s_{x}}}\leq
C\varepsilon_{0}. \label{estimate-stability-integral-3d}%
\end{equation}

\end{lemma}

\begin{proof}
Denote $L_{0}$ to be the linearized operator corresponding to the linearized
Vlasov-Poisson equation at $\left(  f_{0}\left(  v\right)  ,0\right)  $, and
$\mathcal{E}$ is the mapping from $f\left(  x,v\right)  $ to $\vec{E}\left(
x\right)  $ by the Poisson equation (\ref{L-Poisson})$.$It follows from Lemma
\ref{lemma-estimate-linear} that: For any $0\leq s_{v}\leq s_{0}-1,\ $if
$h\left(  x,v\right)  \in H_{x}^{s_{x}}H_{v}^{s_{v},b},$then
\begin{equation}
\left\Vert \left(  1+t\right)  ^{s_{v}}\mathcal{E}\left(  e^{tL_{0}}h\right)
\right\Vert _{L_{t}^{2}H_{x}^{\frac{3}{2}+s_{x}}}\leq C\left\Vert h\left(
x,v\right)  \right\Vert _{H_{x}^{s_{x}}H_{v}^{s_{v},b}}.
\label{estimate-linear-notation}%
\end{equation}
Denote $f_{1}\left(  t\right)  =f\left(  t\right)  -f_{0}$, then
\[
\partial_{t}f_{1}=L_{0}f_{1}+\vec{E}\cdot\partial_{v}f_{1}.
\]
Thus
\[
f_{1}\left(  t\right)  =e^{tL_{0}}f_{1}\left(  0\right)  +\int_{0}%
^{t}e^{\left(  t-u\right)  L_{0}}\left(  \vec{E}\cdot\partial_{v}f_{1}\right)
\left(  u\right)  du=f_{\text{lin}}\left(  t\right)  +f_{\text{non}}\left(
t\right)  ,
\]
and correspondingly
\[
\vec{E}\left(  t\right)  =\mathcal{E}\left(  f_{\text{lin}}\left(  t\right)
\right)  +\mathcal{E}\left(  f_{\text{non}}\left(  t\right)  \right)  =\vec
{E}_{\text{lin}}\left(  t\right)  +\vec{E}_{\text{non}}\left(  t\right)  .
\]
By the linear estimate (\ref{estimate-linear-notation}),%
\begin{align*}
\left\Vert \left(  1+t\right)  ^{s_{v}-1}\vec{E}_{\text{lin}}\left(
x,t\right)  \right\Vert _{L_{\left\{  t\geq0\right\}  }^{2}H_{x}^{\frac{3}%
{2}+s_{x}}}  &  =\left\Vert \left(  1+t\right)  ^{s_{v}-1}\mathcal{E}\left(
e^{tL_{0}}f_{1}\left(  0\right)  \right)  \right\Vert _{L_{t}^{2}H_{x}%
^{\frac{3}{2}+s_{x}}}\\
&  \leq C\left\Vert f_{1}\left(  0\right)  \right\Vert _{H_{x}^{s_{x}}%
H_{v}^{s_{v},b}},
\end{align*}
and
\begin{align*}
&  \left\Vert \left(  1+t\right)  ^{s_{v}-1}\vec{E}_{\text{non}}\left(
x,t\right)  \right\Vert _{L_{\left\{  t\geq0\right\}  }^{2}H_{x}^{\frac{3}%
{2}+s_{x}}}^{2}\\
&  =\int_{0}^{\infty}\left(  1+t\right)  ^{2\left(  s_{v}-1\right)
}\left\Vert \vec{E}_{\text{non}}\left(  x,t\right)  \right\Vert _{H_{x}%
^{\frac{3}{2}+s_{x}}}^{2}dt\\
&  \leq\int_{0}^{\infty}\left(  1+t\right)  ^{2\left(  s_{v}-1\right)
}\left(  \int_{0}^{t}\left\Vert \mathcal{E}\left[  e^{\left(  t-u\right)
L_{0}}\left(  \vec{E}\partial_{v}f_{1}\right)  \left(  u\right)  \right]
\right\Vert _{H_{x}^{\frac{3}{2}+s_{x}}}du\right)  ^{2}dt\\
&  \leq\int_{0}^{\infty}\left(  1+t\right)  ^{2\left(  s_{v}-1\right)  }%
\int_{0}^{t}\left(  1+\left(  t-u\right)  \right)  ^{-2\left(  s_{v}-1\right)
}\left(  1+u\right)  ^{-2\left(  s_{v}-1\right)  }du\\
&  \ \ \ \ \ \ \cdot\int_{0}^{t}\left(  1+u\right)  ^{2\left(  s_{v}-1\right)
}\left(  1+\left(  t-u\right)  \right)  ^{2\left(  s_{v}-1\right)  }\left\Vert
\mathcal{E}\left[  e^{\left(  t-u\right)  L_{0}}\left(  \vec{E}\cdot
\partial_{v}f_{1}\right)  \left(  u\right)  \right]  \right\Vert
_{H_{x}^{\frac{3}{2}+s_{x}}}^{2}\ dudt\\
&  \leq C\int_{0}^{\infty}\int_{0}^{t}\left(  1+u\right)  ^{2\left(
s_{v}-1\right)  }\left(  1+\left(  t-u\right)  \right)  ^{2\left(
s_{v}-1\right)  }\left\Vert \mathcal{E}\left[  e^{\left(  t-u\right)  L_{0}%
}\left(  \vec{E}\cdot\partial_{v}f_{1}\right)  \left(  u\right)  \right]
\right\Vert _{H_{x}^{\frac{3}{2}+s_{x}}}^{2}\ dudt\\
&  =C\int_{0}^{\infty}\left(  1+u\right)  ^{2\left(  s_{v}-1\right)  }\int
_{u}^{\infty}\left(  1+\left(  t-u\right)  \right)  ^{2\left(  s_{v}-1\right)
}\left\Vert \mathcal{E}\left[  e^{\left(  t-u\right)  L_{0}}\left(  \vec
{E}\cdot\partial_{v}f_{1}\right)  \left(  u\right)  \right]  \right\Vert
_{H_{x}^{\frac{3}{2}+s_{x}}}^{2}\ dtdu\\
&  \leq C\int_{0}^{\infty}\left(  1+u\right)  ^{2\left(  s_{v}-1\right)
}\left\Vert \left(  \vec{E}\cdot\partial_{v}f_{1}\right)  \left(  u\right)
\right\Vert _{H_{x}^{s_{x}}H_{v}^{s_{v}-1,b}}^{2}du\\
&  \leq C\int_{0}^{\infty}\left(  1+u\right)  ^{2\left(  s_{v}-1\right)
}\left\Vert \vec{E}\left(  u\right)  \right\Vert _{H_{x}^{\frac{3}{2}+s_{x}}%
}^{2}\left\Vert f_{1}\left(  u\right)  \right\Vert _{H_{x}^{s_{x}}H_{v}%
^{s_{v},b}}^{2}du\\
&  \leq C\varepsilon_{0}^{2}\left\Vert \left(  1+t^{s_{v}-1}\right)  \vec
{E}\left(  x,t\right)  \right\Vert _{L_{\left\{  t\geq0\right\}  }^{2}%
H_{x}^{\frac{3}{2}+s_{x}}}^{2}.
\end{align*}
In the above estimate, we use the fact that
\[
\int_{0}^{t}\left(  1+\left(  t-u\right)  \right)  ^{-2\left(  s_{v}-1\right)
}\left(  1+u\right)  ^{-2\left(  s_{v}-1\right)  }du\leq C\left(  1+t\right)
^{-2\left(  s_{v}-1\right)  }%
\]
because of the assumption that $s_{v}-1>\frac{1}{2}$. The assumption
(\ref{condition-non-existence}) ensures that the following inequality is true
\[
\left\Vert \left(  \vec{E}\cdot\partial_{v}f_{1}\right)  \left(  u\right)
\right\Vert _{H_{x}^{s_{x}}H_{v}^{s_{v}-1,b}}\leq C\left\Vert \vec{E}\left(
u\right)  \right\Vert _{H_{x}^{\frac{3}{2}+s_{x}}}\left\Vert f_{1}\left(
u\right)  \right\Vert _{H_{x}^{s_{x}}H_{v}^{s_{v},b}}.
\]
Thus
\begin{align*}
&  \left\Vert \left(  1+t^{s_{v}-1}\right)  \vec{E}\left(  x,t\right)
\right\Vert _{L_{\left\{  t\geq0\right\}  }^{2}H_{x}^{\frac{3}{2}+s_{x}}}\\
&  \leq\left\Vert \left(  1+t\right)  ^{s_{v}-1}\vec{E}_{\text{lin}}\left(
x,t\right)  \right\Vert _{L_{\left\{  t\geq1\right\}  }^{2}H_{x}^{\frac{3}%
{2}+s_{x}}}+\left\Vert \left(  1+t\right)  ^{s_{v}-1}\vec{E}_{\text{non}%
}\left(  x,t\right)  \right\Vert _{L_{\left\{  t\geq0\right\}  }^{2}%
H_{x}^{\frac{3}{2}+s_{x}}}\\
&  \leq C\left\Vert f_{1}\left(  0\right)  \right\Vert _{H_{x}^{s_{x}}%
H_{v}^{s_{v},b}}+C\varepsilon_{0}\left\Vert \left(  1+t^{s_{v}-1}\right)
\vec{E}\left(  x,t\right)  \right\Vert _{L_{\left\{  t\geq0\right\}  }%
^{2}H_{x}^{\frac{3}{2}+s_{x}}}.
\end{align*}
By taking $\varepsilon_{0}=\frac{1}{2C},$ we get the estimate
(\ref{estimate-stability-integral-3d}).
\end{proof}

Theorem \ref{thm-non-existence} follows from Lemma
\ref{lemma-estimate-integral-small} and the time translation symmetry of the
Vlasov-Poisson equation. Since the arguments are exactly the same as in the
$1$D case (\cite{lin-zeng-bgk}), we skip the details.

As a corollary of Theorem \ref{thm-non-existence}, we get the following
nonlinear instability result.

\begin{corollary}
Assume $f_{0}\left(  v\right)  \in H^{s_{0},b}\left(  \mathbf{R}^{d}\right)  $
$\left(  d\geq2,\ s_{0}>\frac{3}{2},\ b>\frac{d-1}{4}\right)  \ $and the
Penrose stability condition (\ref{penrose-condition}) is satisfied for the
$x-$period tuple $\left(  T_{1},\cdots,T_{d}\right)  $. For any $\left(
s_{x},s_{v}\right)  $ satisfying (\ref{condition-non-existence}), there exists
$\varepsilon_{0}>0$ such that for any solution $\left(  f\left(  x,v,t\right)
,\vec{E}\left(  x,t\right)  \right)  $ of the Vlasov-Poisson system
(\ref{vlasov})-(\ref{poisson}) with $x-$period tuple $\left(  T_{1}%
,\cdots,T_{d}\right)  $ and $\vec{E}\left(  x,0\right)  $ not identically
zero, the following is true: \
\[
\left\Vert f\left(  T^{\ast}\right)  -f_{0}\right\Vert _{H_{x}^{s_{x}}%
H_{v}^{s_{v},b}}\geq\varepsilon_{0},\ \text{for\ some\ }T^{\ast}\in
\mathbf{R}.
\]

\end{corollary}

We can also study the positive (negative) invariant structures near $\left(
f_{0}\left(  v\right)  ,0\right)  ,$ which are solutions $\left(  f\left(
t\right)  ,\vec{E}\left(  t\right)  \right)  $ of nonlinear Vlasov-Poisson
equation satisfying the conditions (\ref{assumption-thm-invariant-2d}%
)$\ $for\ all\ $t\geq0$ ($t\leq0$)$.$ The next theorem shows that the electric
field of these semi-invarint structures must decay when $t\rightarrow+\infty$
($t\rightarrow-\infty).$

\begin{theorem}
\label{thm-semi-invariant}Assume the homogeneous profile
\[
f_{0}\left(  v\right)  \in H^{s,b}\left(  \mathbf{R}^{d}\right)
\ \ \ \ \left(  d\geq2,\ s_{0}>\frac{3}{2},\ b>\frac{d-1}{4}\right)  .
\]
$\ $Assume that $f_{0}\left(  v\right)  $ satisfies the Penrose stability
condition (\ref{penrose-condition}) for $\left(  T_{1},\cdots,T_{d}\right)  $.
Let $\left(  f\left(  x,v,t\right)  ,\vec{E}\left(  x,v,t\right)  \right)  $
be a solution of (\ref{vpe}) in $T^{d}.$

For any $\left(  s_{x},s_{v}\right)  $ satisfying
(\ref{condition-non-existence}),$\ $there exists $\varepsilon_{0}>0$, such
that if
\[
\left\Vert f\left(  t\right)  -f_{0}\right\Vert _{H_{x}^{s_{x}}H_{v}^{s_{v}%
,b}}<\varepsilon_{0},\ \text{for\ all\ }t\geq0\text{ }\left(  \text{or }%
t\leq0\right)  ,
\]
with$\ $
\[
\left\Vert f\left(  0\right)  \right\Vert _{L_{x,v}^{\infty}}<\infty
,\ \int_{T^{d}}\int_{\mathbf{R}^{d}}\left\vert v\right\vert ^{2}f\left(
0,x,v\right)  dvdx<\infty,
\]
then $\left\Vert \vec{E}\left(  t,x\right)  \right\Vert _{L_{x}^{2}%
}\rightarrow0$ when $t\rightarrow+\infty$ $\left(  \text{or }t\rightarrow
-\infty\right)  .$
\end{theorem}

\begin{proof}
By energy conservation,
\begin{align*}
&  \int_{T^{d}}\int_{\mathbf{R}^{d}}\left\vert v\right\vert ^{2}f\left(
x,v,t\right)  dvdx+\left\Vert \vec{E}\left(  x,t\right)  \right\Vert
_{L_{x}^{2}}^{2}\\
&  =\int_{T^{d}}\int_{\mathbf{R}^{d}}\left\vert v\right\vert ^{2}f\left(
x,v,0\right)  dvdx+\left\Vert E\left(  x,0\right)  \right\Vert _{L^{2}}^{2}<C.
\end{align*}
Let $j=\int vf\ dv$. When $d=2,$we have
\begin{align*}
\left\vert j\left(  t\right)  \right\vert  &  =\left\vert \int vf\ \left(
t\right)  dv\right\vert \leq\int_{\left\vert v\right\vert \leq A}\left\vert
v\right\vert dv\left\Vert f\left(  t\right)  \right\Vert _{L_{x,v}^{\infty}%
}+\frac{1}{A}\int_{\left\vert v\right\vert \geq A}\left\vert v\right\vert
^{2}fdv\\
&  \leq C\left(  \left\Vert f\left(  0\right)  \right\Vert _{L_{x,v}^{\infty}%
}A^{3}+\frac{1}{A}\int\left\vert v\right\vert ^{2}fdv\right)  \leq C\left\Vert
f\left(  0\right)  \right\Vert _{L_{x,v}^{\infty}}^{\frac{1}{4}}\left(
\int\left\vert v\right\vert ^{2}fdv\right)  ^{\frac{3}{4}},
\end{align*}
by choosing
\[
A=\left(  \int\left\vert v\right\vert ^{2}fdv/\left\Vert f\left(  0\right)
\right\Vert _{L_{x,v}^{\infty}}\right)  ^{\frac{1}{4}}.
\]
Thus
\[
\left\Vert j\left(  x,t\right)  \right\Vert _{L_{x}^{\frac{4}{3}}}\leq
C\int\int\left\vert v\right\vert ^{2}f\ dvdx\leq C.
\]
Since%
\begin{align*}
\frac{d}{dt}\left\Vert \vec{E}\left(  x,t\right)  \right\Vert _{L_{x}^{2}%
}^{2}  &  =\int_{T^{d}}j\left(  x,t\right)  \cdot\vec{E}\left(  x,t\right)
dx\\
&  \leq\left\Vert j\left(  x,t\right)  \right\Vert _{L_{x}^{\frac{4}{3}}%
}\left\Vert E\left(  x,t\right)  \right\Vert _{L_{x}^{4}}\leq C\left\Vert
E\left(  x,t\right)  \right\Vert _{H_{x}^{\frac{3}{2}}},
\end{align*}
and by Lemma \ref{lemma-estimate-integral-small}
\begin{align*}
\int_{0}^{\infty}\left\Vert E\left(  x,t\right)  \right\Vert _{H_{x}^{\frac
{3}{2}}}dt  &  \leq\left(  \int_{0}^{\infty}\left(  1+t\right)  ^{-2\left(
s-1\right)  }dt\right)  ^{\frac{1}{2}}\left(  \int_{0}^{\infty}\left(
1+t\right)  ^{2\left(  s-1\right)  }\left\Vert E\left(  x,t\right)
\right\Vert _{H_{x}^{\frac{3}{2}}}^{2}dt\right)  ^{\frac{3}{2}}\\
&  \leq C\varepsilon_{0},
\end{align*}
thus $\lim_{t\rightarrow\infty}\left\Vert \vec{E}\left(  x,t\right)
\right\Vert _{L_{x}^{2}}$ exists and equals zero. When $d=3$, the proof is
very similar. The estimates become
\[
\left\Vert j\left(  x,t\right)  \right\Vert _{L_{x}^{\frac{5}{4}}}\leq C,
\]
and%
\[
\frac{d}{dt}\left\Vert \vec{E}\left(  x,t\right)  \right\Vert _{L_{x}^{2}}%
^{2}\leq\left\Vert j\left(  x,t\right)  \right\Vert _{L_{x}^{\frac{5}{4}}%
}\left\Vert E\left(  x,t\right)  \right\Vert _{L_{x}^{5}}\leq C\left\Vert
E\left(  x,t\right)  \right\Vert _{H_{x}^{\frac{3}{2}+s_{x}}}.
\]
The rest is the same.
\end{proof}

\begin{center}
\bigskip

{\Large Acknowledgement}
\end{center}

This work is supported partly by the NSF grants DMS-0908175 (Lin) and
DMS-0801319 (Zeng).

\end{document}